\documentclass[11pt,reqno,a4paper]{amsart}

\usepackage{etoolbox}

\makeatletter
\patchcmd{\@setaddresses}{\indent}{\noindent}{}{}
\patchcmd{\@setaddresses}{\indent}{\noindent}{}{}
\patchcmd{\@setaddresses}{\indent}{\noindent}{}{}
\patchcmd{\@setaddresses}{\indent}{\noindent}{}{}
\makeatother

\usepackage[utf8]{inputenc}

\usepackage{amssymb}
\usepackage{amsmath}
\usepackage{amsthm}
\usepackage{latexsym}
\usepackage{amsfonts}
\usepackage{titlesec}
\usepackage{color}
\usepackage{url}
\usepackage{dsfont}
\usepackage{thmtools}
\usepackage{thm-restate}
\usepackage[unicode=true]{hyperref}

\newtheorem{thm}{Theorem}[section]
\newtheorem{cor}[thm]{Corollary}
\newtheorem{lem}[thm]{Lemma}
\newtheorem{prop}[thm]{Proposition}

\theoremstyle{definition}

\newtheorem{rem}[thm]{Remark}
\newtheorem{ex}[thm]{Example}

\newtheorem*{rem*}{Remark}

\titleformat{\section}{\normalfont\bfseries\centering}{\thesection.}{.25em}{}
\titleformat{\subsection}{\normalfont\bfseries}{\thesubsection.}{.25em}{}
\numberwithin{equation}{section}
\allowdisplaybreaks

\newcommand{\Indicator}{\mathds{1}}
\newcommand{\one}{\mathbf{1}}
\newcommand{\braces}[1]{{\rm (}#1{\rm )}}
\newcommand{\rmref}[1]{{\rm\ref{#1}}}

\newcommand{\Fourier}{\mathcal{F}}

\binoppenalty=\maxdimen
\relpenalty=\maxdimen
\newcommand{\R}{\ensuremath{\mathbb R}}    
\newcommand{\CC}{\ensuremath{\mathbb C}}   
\newcommand{\Q}{\ensuremath{\mathbb Q}}    
\newcommand{\N}{\ensuremath{\mathbb N}}    
\newcommand{\Z}{\ensuremath{\mathbb Z}}    

\newcommand{\iso}{\circ}

\newcommand{\<}{\langle}
\renewcommand{\>}{\rangle}

\newcommand{\calA}{\mathcal A}         
\newcommand{\calB}{\mathcal B}

\newcommand{\calF}{\mathcal F}         
\newcommand{\calG}{\mathcal G}         
\newcommand{\calH}{\mathcal H}         
         \newcommand{\frakI}{\mathfrak I}
         
\newcommand{\calK}{\mathcal K}         
\newcommand{\calL}{\mathcal L}

\newcommand{\calS}{\mathcal S}

\newcommand{\la}{\lambda}

\newcommand{\vphi}{\varphi}

\newcommand{\bmat}[4]
{
   \begin{bmatrix}
      #1 & #2\\
      #3 & #4
   \end{bmatrix}
}

\newcommand{\smallmat}[4]{\left(\begin{smallmatrix}#1 & #2\\#3 & #4\end{smallmatrix}\right)}
\newcommand{\sbmat}[4]{\left[\begin{smallmatrix}#1 & #2\\#3 & #4\end{smallmatrix}\right]}

\newcommand{\linspan}{\operatorname{span}}

\newcommand{\ran}{\operatorname{ran}}

\newcommand{\Llra}{\Longleftrightarrow}

\newcommand{\ol}{\overline}

\newcommand{\directSum}{\boxplus}

\newcommand{\wh}{\widehat}

\newcommand{\Inv}{\mathfrak I}

\newcommand{\diag}{\operatorname{diag}}
\newcommand{\La}{\Lambda}

\newcommand{\HH}{\mathbb{H}}
\newcommand{\WW}{\mathbb{W}}
\renewcommand{\setminus}{\backslash}
\newcommand{\topp}{{\!\top}}

\begin{document}
\title[]{A Balian--Low type theorem\\{} for Gabor Riesz sequences of arbitrary density}

\author[A. Caragea]{Andrei Caragea}
\address{{\bf A.~Caragea:} KU Eichst\"att--Ingolstadt, Mathematisch--Geographische Fakult\"at,
Ostenstra\ss{}e 26, Kollegiengeb\"aude I Bau B, 85072 Eichst\"att, Germany}
\email{andrei.caragea@gmail.com}

\author[D.G. Lee]{Dae Gwan Lee}
\address{{\bf D.G.~Lee:} KU Eichst\"att--Ingolstadt, Mathematisch--Geographische Fakult\"at,
Os\-ten\-stra\-\ss{}e 26, Kollegiengeb\"aude I Bau B, 85072 Eichst\"att, Germany}
\email{daegwans@gmail.com}

\author[F. Philipp]{Friedrich Philipp}
\address{{\bf F.~Philipp:} Technische Universit\"at Ilmenau, Institute for Mathematics, Weimarer Stra\ss e 25, D-98693 Ilmenau, Germany}
\email{friedrich.philipp@tu-ilmenau.de}

\author[F. Voigtlaender]{Felix Voigtlaender}
\address{{\bf F.~Voigtlaender:}
KU Eichstätt-Ingolstadt,
Lehrstuhl Reliable Machine Learning,
Ostenstraße 26,
85072 Eichstätt,
Germany}
\email{felix.voigtlaender@ku.de}


\begin{abstract}
We consider Gabor Riesz sequences generated by a lattice $\Lambda \subset \R^2$ and a window
function $g \in L^2(\R)$ which is well localized in both time and frequency.
When $g$ belongs to the Feichtinger algebra,
we prove that only those time-frequency shifts with parameters from the lattice $\Lambda$
leave the corresponding Gabor space invariant.
This improves on earlier results where only lattices of rational density were considered.
A slightly weaker result is proved---again for lattices of general density---%
under the regularity assumptions of the classical Balian-Low theorem,
where both $g$ and its Fourier transform belong to the Sobolev space $H^1(\R)$.
The proof relies on a combination of methods from time-frequency analysis
and the theory of $C^\ast$-algebras, specifically the so-called irrational rotation algebra.
\vspace*{-0.45cm}
\end{abstract}

\subjclass[2010]{Primary: 42C15. Secondary: 42C30, 42C40.}
\keywords{Gabor systems;
Riesz sequences;
Time-frequency shifts;
Balian-Low theorem;
Ron-Shen duality;
Completeness of Gabor systems;
Feichtinger algebra;
Irrational rotation algebra.}

\maketitle
\thispagestyle{empty}

\section{Introduction}

When working with Gabor frames, the window function $g$ should have a good
time-frequency localization, so that the frame coefficients faithfully reflect
the time-frequency behavior of the analyzed function.
The \emph{Feichtinger algebra} $S_0(\R^d)$
\cite{FeichtingerNewSegalAlgebra,JakobsenNoLongerNewSegalAlgebra}
is a particularly popular window class.
Among other advantages, choosing a window from $S_0$ ensures that the canonical dual window
also belongs to the Feichtinger algebra \cite{GroechenigLeinert},
so that for example the membership of a function $f \in L^2(\R^d)$
in the modulation space $M^{p,q}(\R^d)$ can be characterized
in terms of the decay properties of its frame coefficients.
One crucial obstruction, however, is that a Gabor system with window belonging to $S_0(\R^d)$
can not form an orthonormal basis---in fact not even a Riesz basis---for $L^2(\R^d)$.
We call this phenomenon the \emph{$S_0$ Balian-Low theorem};
it is a consequence of the Amalgam Balian-Low theorem
\cite[Theorem~3.2]{BenedettoDifferentiationAndBLT}. The same no-go type result holds for the case where $g$ belongs to the space
$\HH^1(\R^d)$ consisting of functions in the $L^2$-Sobolev space $H^1(\R^d)$
whose Fourier transform also belongs to $H^1(\R^d)$.
This is the classical Balian-Low theorem; see \mbox{\cite[Theorem~8.4.5]{GroechenigTFFoundations}}
for the case of orthonormal bases, and \cite[Theorem~2.3]{DaubechiesBalianLow} for the general case.

Yet, even though a Gabor system with $g \in S_0(\R^d)$ cannot form a Riesz basis
for all of $L^2(\R^d)$, it might still be a \emph{Riesz sequence}, that is,
a Riesz basis for its closed linear span $\calG (g, \Lambda)$,
at least if $\calG(g,\Lambda)$ is a \emph{proper} subspace of $L^2(\R^d)$.
In this case, one might wonder about further properties---in addition
to being a proper subspace of $L^2(\R^d)$---that the Gabor space $\calG(g,\Lambda)$ has to have.
One important property in time-frequency analysis is the invariance of $\calG(g,\Lambda)$
under time-frequency shifts $T_a M_b$.
For lattices $\Lambda$ \emph{of rational density} and for dimension $d = 1$,
it was observed in \cite{TFInvarianceAndAmalgamBL} that if $(g,\Lambda)$ is a Riesz
sequence and if $g \in S_0(\R)$, then the set of parameters $(a,b) \in \R^2$
such that $\calG(g,\Lambda)$ is invariant under the time-frequency shift $T_a M_b$ is
exactly equal to $\Lambda$.
A multi-dimensional variant of this
was derived in \cite{AmalgamBLForSymplecticLatticesRationalDensity}.

These results generalize the $S_0$ Balian-Low theorem to subspaces of $L^2(\R)$.
Indeed, to derive the $S_0$ Balian-Low theorem from the above result,
note that if ${\calG(g, \Lambda) = L^2(\R)}$ then
$\calG(g,\Lambda)$ is invariant under \emph{all} time-frequency shifts,
even under those with $(a,b) \notin \Lambda$; hence, $g$ cannot belong to $S_0$.
A corresponding generalization of the \emph{classical} Balian-Low theorem was proved
in \cite{BLForSubspaces}; a quantitative version can be found in \cite{QuantitativeSubspaceBL}.

We emphasize that in all articles
\cite{TFInvarianceAndAmalgamBL,AmalgamBLForSymplecticLatticesRationalDensity,
BLForSubspaces,QuantitativeSubspaceBL} it is assumed that the generating lattice $\Lambda$
has rational density.
This restriction is needed in order to utilize the Zak transform which is used extensively in
\cite{TFInvarianceAndAmalgamBL,AmalgamBLForSymplecticLatticesRationalDensity,
BLForSubspaces,QuantitativeSubspaceBL}.
It is thus natural to ask whether the results in \cite{TFInvarianceAndAmalgamBL}
and \cite{BLForSubspaces} still hold for lattices with \emph{irrational} density.

In a sense, this question has analogies with the research concerning the regularity
of the canonical dual window of a Gabor frame.
In 1997 it was shown
(see \cite[Theorem~3.4]{FeichtingerGroechenigGaborFramesAndTFAnalysisDistributions})
that if $g\in S_0(\R^d)$ generates a Gabor frame for $L^2(\R^d)$
over a lattice \emph{of rational density}, then the canonical dual window
also belongs to $S_0(\R^d)$.
It was conjectured in the same article that this property
continues to hold for general lattices.
Six years later, this conjecture was confirmed by Gröchenig and Leinert \cite{GroechenigLeinert}
by using $C^*$-algebra methods.

Here, we likewise extend the result in \cite{TFInvarianceAndAmalgamBL} to arbitrary lattices:

\begin{restatable}{thm}{MainTheorem}
\label{thm:GeneralizedBLConjecture_Intro}
If $g \in S_0(\R)$ and $\Lambda \subset \R^2$ is a lattice such that the Gabor system $(g, \La)$ is
a Riesz basis for its closed linear span $\calG(g, \La)$, then the time-frequency shifts $T_a M_b$
that leave $\calG(g,\La)$ invariant satisfy $(a,b) \in \Lambda$.
\end{restatable}

As indicated above, the Zak transform is a powerful tool for analyzing Gabor systems
generated by lattices with rational density; yet, it is not of much use
in the case of irrational density lattices.
Consequently, the methods used in the present paper differ substantially from those in
\cite{TFInvarianceAndAmalgamBL,AmalgamBLForSymplecticLatticesRationalDensity,
BLForSubspaces,QuantitativeSubspaceBL}:
Instead of applying the Zak transform and thus dealing with functions on $\R^2$,
we work directly with the given objects and exploit the rich theory of time-frequency analysis.
Along the way, we obtain several new statements related to time-frequency shift
invariance that are interesting in their own right.

The proof of Theorem~\ref{thm:GeneralizedBLConjecture_Intro} consists of several steps.
First, for $g \in S_0(\R)$ and only assuming that $(g,\Lambda)$ is
a \emph{frame sequence}---that is, a frame for its closed linear span---%
we prove the following dichotomy:
\begin{equation}
  \parbox{0.8\textwidth}{
    \it Either $(g,\Lambda)$ spans all of $L^2(\R)$, or the set of $(a,b) \in \R^2$
    for which $T_aM_b$ leaves $\calG(g,\La)$ invariant
    is a lattice containing $\La$ as a sublattice;
  }
  \label{eq:IntroductionDichotomy}
  \tag{D}
\end{equation}
see Theorem \ref{thm:DichotomyForNiceWindows}.
This result significantly reduces the range of parameters $(a,b)$ that we need to consider.
Next, we give a characterization
for the invariance of $\calG(g,\La)$ under a time-frequency shift $T_aM_b$ with $(a,b) \notin \La$
in terms of the adjoint system of $(g, \Lambda)$;
see Theorem~\ref{thm:AdditionalTFShiftAdjointCharacterization}.
This characterization holds for general $g \in L^2(\R)$, not only for $g \in S_0(\R)$.
Combining this characterization with a deep existing result about traces of projections
in the so-called \emph{irrational rotation algebra} (see
\cite{PimsnerVoiculescuIrrationalRotationAlgebra,RieffelIrrationalRotationAlgebra}), we arrive
at the conclusion of Theorem~\ref{thm:GeneralizedBLConjecture_Intro}.

With Theorem~\ref{thm:GeneralizedBLConjecture_Intro} established for $g$ in the Feichtinger algebra,
it is natural to ask whether the same statement holds in the setting of the
\emph{classical} Balian-Low theorem, that is, when $g$ has finite uncertainty product
\(
  \bigl(
    \int x^2 |g(x)|^2 \,dx
  \bigr) \cdot
  \bigl(
    \int \omega^2 |\widehat{g}(\omega)|^2 \, d\omega
  \bigr)
  < \infty
  ,
\)
a condition which we simply write as $g \in \HH^1$.
Unfortunately, we were not able to prove a full-fledged version
of Theorem~\ref{thm:GeneralizedBLConjecture_Intro} for $g \in \HH^1$;
the best we could do is to show that the dichotomy \eqref{eq:IntroductionDichotomy}
described above for $g \in S_0$ still holds for $g \in \HH^1$.

The outline of the paper is as follows:
After recalling the necessary background on Janssen's representation,
time-frequency shift invariance, symplectic operators, and the two spaces $S_0(\R)$ and $\HH^1$
in Section~\ref{sec:Preliminaries},
the paper proper  starts in Section~\ref{sec:TFInvarianceForFrames},
where we prove the dichotomy \eqref{eq:IntroductionDichotomy} described above,
for $g \in S_0(\R) + \HH^1$.
Next, in Section~\ref{sec:TFInvariantAndFrameSequences} we show that one can reduce
to the case of a separable lattice $\Lambda = \alpha \Z \times \beta \Z$,
with an additional time-frequency shift of the form $T_{\alpha/\nu}$ for some $\nu \in \N_{\geq 2}$.
For this setting, we then derive a characterization in terms of the adjoint Gabor system.
Throughout Section~\ref{sec:TFInvariantAndFrameSequences},
the generating function $g$ is only assumed to be in $L^2(\R)$.
The paper culminates in Section~\ref{sec:S0}, where we
prove Theorem~\ref{thm:GeneralizedBLConjecture_Intro}.
Finally, Appendix~\ref{sec:Appendix} contains a short treatise
on the irrational rotation algebra and a corresponding result that is crucial for our
proof of Theorem~\ref{thm:GeneralizedBLConjecture_Intro}.

\section{Preliminaries}
\label{sec:Preliminaries}

For $a, b \in \R$ and $f \in L^2(\R)$ we define the operators of translation
by $a$ and modulation by $b$ as
\vspace{-0.3cm}
\[
  T_a f(x) := f(x-a)
  \quad \text{and} \quad
  M_b f(x) := e^{2\pi ib x} f(x),
\]
respectively. Both $T_a$ and $M_b$ are unitary operators on $L^2(\R)$
and hence so is the {\em time-frequency shift}
\[
  \pi(a,b)
  := T_a M_b
   = e^{-2\pi iab} \, M_b T_a .
\]
A \emph{lattice} $\Lambda \subset \R^2$ is any set of the form $\Lambda = A \Z^2$
with an invertible matrix $A \in \R^{2 \times 2}$.
The {\em density} of $\Lambda$ is defined by $d(\Lambda) = |\!\det A|^{-1}$.
Note that $A\Z^2 = \Z^2$ if and only if $A\in\Z^{2\times 2}$ and $\det A = \pm 1$.
This will be used heavily in the proof of Proposition~\ref{p:trichotomy} below.

A lattice $\Lambda$ is called {\em separable} if $A$ can be chosen to be diagonal,
i.e., $\Lambda = \alpha\Z\times\beta\Z$ with $\alpha,\beta > 0$.
The next lemma shows that every lattice can be transformed into a separable one
by means of a symplectic matrix; this will be used frequently.

\begin{lem}\label{lem:SymplecticLatticeSeparation}
  Let $A \in \R^{2 \times 2}$ be a non-singular matrix.
  Then there exists $C \in \R^{2 \times 2}$ with $\det C = 1$ such that $CA$ is diagonal,
  i.e., $CA\Z^2$ is separable.
\end{lem}

\begin{proof}
  Write $A = \smallmat{a}{b}{c}{d}$ and note $\Delta := ad - bc \neq 0$.
  If $a \neq 0$, choose $C = \smallmat{1+bc/\Delta}{-ab/\Delta}{-c/a}{1}$.
  Then a simple calculation yields $\det C = 1$ and $CA = \mathrm{diag} (a, \Delta/a)$.
  In the case $a = 0$ we have $b \neq 0 \neq c$ as $A$ is non-singular.
  Then $C := \smallmat{-d/b}{1}{-1}{0}$ satisfies $\det C = 1$ and $CA = \mathrm{diag} (c, -b)$.
\end{proof}

For a subset $M \subset L^2(\R)$, we denote its closure by $\overline{M}$.
Then, for $g \in L^2(\R)$ and a lattice $\Lambda\subset\R^2$ we set
\[
  (g,\Lambda) := \bigl\{ \pi(\la) g : \la \in \Lambda \bigr\}
  \qquad \text{and} \qquad
  \calG(g,\Lambda) := \ol{\vphantom{t} \linspan}\,(g,\Lambda) \subset L^2(\R).
\]

For the Fourier transform, we use the normalization
\(
  \Fourier f (\xi)
  = \widehat{f}(\xi)
  = \int_\R
      f(x) \, e^{-2 \pi i x \xi}
    \, d x
\)
for $f \in L^1(\R)$.
It is well-known that $\Fourier$ extends to a unitary map $\Fourier : L^2(\R) \to L^2(\R)$.

\subsection{Bessel vectors and \texorpdfstring{Janssen's}{Janssen’s} representation}

Let $\La = \alpha\Z\times\beta\Z$ be a separable lattice with $\alpha,\beta > 0$.
The \emph{adjoint lattice} of $\La$ is defined as
$\La^\iso = \tfrac{1}{\beta} \Z \times \tfrac{1}{\alpha}\Z$.
We say that $g \in L^2(\R)$ is a {\em Bessel vector} for $\La$ if the system $(g,\La)$
is a Bessel system in $L^2(\R)$, meaning that the {\em analysis operator} $C_{\La,g}$
corresponding to $(g,\La)$ is bounded as an operator from $L^2(\R)$ to $\ell^2(\Z^2)$.
It is defined by
\[
  C_{\La,g} f = \big(
                  \< f, T_{m\alpha} M_{n\beta} g \>
                \big)_{m,n\in\Z},
  \qquad f \in L^2(\R).
\]
We denote the set of Bessel vectors for $\La$ by $\calB_\La$.
This is a linear subspace of $L^2(\R)$ which is dense
because it contains the Schwartz space $\calS(\R)$;
see \cite[Corollary~6.2.3]{GroechenigTFFoundations}.
It is well known that $\calB_\La = \calB_{\La^\iso}$
(see \cite[Theorem~2.2(a)]{RonShenDuality}) and that
\begin{equation}\label{e:weak_Janssen}
  \sum_{m, n\in\Z}
    \<f, T_{m\alpha} M_{n\beta} g\>
    \big\< T_{m\alpha} M_{n\beta} h, u \big\>
  = \frac{1}{\alpha\beta}
    \sum_{k,\ell\in\Z}
      \big\< h, T_{\frac k\beta} M_{\frac\ell\alpha} g \big\>
      \big\< T_{\frac k\beta} M_{\frac\ell\alpha} f, u \big\>
\end{equation}
whenever at least three of $f,g,h,u\in L^2(\R)$ are Bessel vectors for $\La$;
this follows from \cite[Proposition~2.4]{j}.
Formula~\eqref{e:weak_Janssen} yields a useful representation
(the so-called \emph{Janssen representation})
of the cross frame operator $S_{\Lambda,g,h} : L^2(\R) \to L^2(\R)$
associated to Bessel vectors $g, h \in \mathcal{B}_\Lambda$.
This operator is defined by
\begin{equation}
  S_{\La,g,h}f
  := \sum_{m,n\in\Z}
       \<f, T_{m\alpha} M_{n\beta} g\>
       \cdot T_{m\alpha} M_{n\beta} h,
  \qquad f \in L^2(\R).
  \label{eq:FrameOperatorDefinition}
\end{equation}
Equation~\eqref{e:weak_Janssen} implies that
\begin{equation}\label{e:3bessel}
  S_{\La,g,h}f
  = \frac{1}{\alpha\beta}
    \sum_{k,\ell\in\Z}
      \big\<h,T_{\frac k\beta}M_{\frac\ell\alpha}g\big\>
      \cdot T_{\frac k\beta}M_{\frac\ell\alpha} f
  \qquad \text{if} \quad f,g,h \in \calB_\La
  .
\end{equation}
The series in Equations~\eqref{eq:FrameOperatorDefinition} and \eqref{e:3bessel}
both converge unconditionally in $L^2(\R)$.

\subsection{Time-frequency shift invariance}

For a closed linear subspace $\calG\subset L^2(\R)$, we denote by $\Inv(\calG)$
the set of all pairs ${(a,b) \in \R^2}$ such that $\calG$ is invariant under the
time-frequency shift $\pi(a,b)$; that is,
\[
  \Inv(\calG)
  := \big\{
       z\in\R^2
       :
       \pi(z)\calG\subset\calG
     \big\}.
\]
If $\calG = \calG(g,\Lambda)$ for some $g\in L^2(\R)$ and a lattice $\Lambda\subset\R^2$,
then clearly $\Lambda \subset \frakI(\calG)$.
Any time-frequency shift $\pi(z)$ with $z \in \frakI(\calG) \backslash \Lambda$ will be called an
{\em additional time-frequency shift} for $\calG(g,\Lambda)$.
For Gabor spaces $\calG = \calG(g,\Lambda)$, the set $\Inv(\calG)$ has some additional structure:

\begin{lem}[{\cite[Proposition~A.1]{TFInvarianceIntegerLattices}}]\label{lem:StabiliserIsGroup}
  Let $g \in L^2(\R)$, let $\Lambda \subset \R^2$ be a lattice,
  and define ${\calG := \calG(g, \Lambda)}$.
  If $z \in \R^2$, then $z\in \Inv(\calG)$ if and only if $\pi(z)g \in \calG$.
  Moreover, $\Inv(\calG)$ is a closed additive subgroup of $\R^2$.
\end{lem}

Lemma~\ref{lem:StabiliserIsGroup} shows that $z\in\Inv(\calG)$ implies $-z\in\Inv(\calG)$,
i.e., $\pi(z)\calG\subset\calG$ and $\pi(z)^{-1}\calG\subset\calG$.
Hence, we have $\pi(z)\calG = \calG$ whenever $z\in\Inv(\calG)$.

The next lemma characterizes the case when $\calG$ is invariant
under \emph{all} time-frequency shifts.

\begin{lem}\label{l:whole_space}
For a closed linear subspace $\calG\subset L^2(\R)$, $\calG\neq\{0\}$,
we have $\frakI(\calG) = \R^2$ if and only if $\calG = L^2(\R)$.
\end{lem}
\begin{proof}
Clearly, if $\calG = L^2(\R)$, then $\frakI(\calG) = \R^2$.
Conversely, assume that $\Inv(\calG) = \R^2$
and let $f\in\calG^{\perp}$ and $g\in\calG\backslash\{0\}$.
Then $\langle f, \pi(z)g \rangle = 0$ for all $z\in\R^2$, so that
the short-time Fourier transform $V_g f$ of $f$ with window $g$ satisfies $V_gf\equiv 0$.
By \cite[Corollary~3.2.2]{GroechenigTFFoundations} and since $g\neq 0$, this implies $f = 0$.
We have thus shown $\calG^{\perp} = \{0\}$,
whence $\calG = L^2(\R)$, since $\calG$ is a closed subspace of $L^2(\R)$.
\end{proof}

\subsection{Symplectic operators}
\label{sub:SymplecticOperators}

It is often useful to reduce a statement involving a non-separable lattice to one that
involves a separable lattice, since separable lattices are usually easier to handle.
For this reduction, we will use so-called \emph{symplectic operators}
(see \cite[Section~9.4]{GroechenigTFFoundations}).
Since we are working in dimension $d = 1$, a matrix $B \in \R^{2 \times 2}$
is \emph{symplectic} if and only if $\det B = 1$; see \cite[Lemma~9.4.1]{GroechenigTFFoundations}.
For any such matrix $B$, it is shown in \cite[Equation~(9.39)]{GroechenigTFFoundations} that
there exists a unitary operator $U_B : L^2(\R) \to L^2(\R)$ such that
\begin{equation}\label{eq:SymplecticOperatorDefinition}
  U_B \rho(z) = \rho(Bz) U_B,
  \quad z \in \R^2,
\end{equation}
where (as in \cite[Page~185 and Equation~(9.25)]{GroechenigTFFoundations})
\[
  \rho(a,b) := e^{\pi iab} \cdot \pi(a,b).
\]
In the sequel, we fix for each $B \in\R^{2\times 2}$ with $\det B = 1$ one choice
of the operator $U_B$, and for functions $g\in L^2(\R)$, closed subspaces $\calG\subset L^2(\R)$,
and sets $\Lambda \subset \R^2$ we write
\[
  g_B := U_B \, g,
  \qquad
  \calG_B := U_B \, \calG,
  \qquad \text{and} \qquad
  \Lambda_B := B \Lambda.
\]
As shown in \cite[Page~197]{GroechenigTFFoundations},
given $B,C \in \R^{2 \times 2}$ with $\det B = \det C = 1$, we have
$U_B U_C = \theta_{B,C} \, U_{B C}$ for some $\theta_{B,C} \in \CC$ with $|\theta_{B,C}| = 1$.

Note that \eqref{eq:SymplecticOperatorDefinition} implies
\begin{equation}\label{eq:SymplecticOperatorInvariantSetTransformation}
  \pi(z)g\in\calG
  \;\Llra\,
  \pi(Bz)g_B\in\calG_B,
  \qquad z\in\R^2.
\end{equation}
Therefore, $(g,\Lambda)$ is a frame (Riesz basis, resp.)
for its closed linear span $\calG$ if and only if $(g_B,\Lambda_B)$ is a frame
(Riesz basis, resp.) for its closed linear span $\calG_B$.
Thanks to Lemma~\ref{lem:StabiliserIsGroup},
the equivalence \eqref{eq:SymplecticOperatorInvariantSetTransformation} also implies that
\begin{equation}\label{eq:InvariantSetSymplecticTransform}
  \Inv(\calG_B) = B\,\Inv(\calG).
\end{equation}

\subsection{The Feichtinger algebra}

We denote by $S_0(\R)$ the {\em Feichtinger algebra}, which is the space
of functions ${f \in L^2(\R)}$ such that $\<f,\pi(\cdot)\vphi\>\in L^1(\R^2)$
for some (and hence every; see \mbox{\cite[Proposition~12.1.2]{GroechenigTFFoundations}})
Schwartz function $\vphi \neq 0$.

Recall that $S_0(\R)$ is invariant under each operator $U_B$
(cf.~\cite[Proposition~12.1.3]{GroechenigTFFoundations}),
so that ${g \in S_0(\R)}$ always implies $g_B \in S_0(\R)$ for $B\in\R^{2\times 2}$
with $\det B = 1$.
Also, each $g \in S_0(\R)$ is a Bessel vector for any (separable) lattice
(see e.g.~\cite[Propositions~6.2.2 and 12.1.4]{GroechenigTFFoundations}).
Since for $g,h\in S_0(\R)$ and any $\alpha,\beta > 0$ the sequence
$\bigl(\< h, T_{m\alpha} M_{n\beta} g\>\bigr)_{m, n \in \Z}$ belongs to $\ell^1(\Z^2)$
(see \mbox{\cite[Corollary~12.1.12]{GroechenigTFFoundations}}),
it follows from \eqref{e:3bessel} and from the density of $\calB_\Lambda$ in $L^2(\R)$ that
\begin{equation}\label{e:S0_CFO}
  S_{\Lambda,g,h}
  = \frac{1}{\alpha\beta}
    \sum_{k,\ell\in\Z}
      \big\<h,T_{\frac k\beta}M_{\frac\ell\alpha}g\big\>
      \cdot T_{\frac k\beta}M_{\frac\ell\alpha}
  \quad \text{with} \quad \Lambda = \alpha \Z \times \beta \Z ,
\end{equation}
where the series converges absolutely in operator norm.

\subsection{The space \texorpdfstring{$\HH^1$}{IH¹}}

Let $H^1(\R)$ denote the space of all functions $f$ in $L^2(\R)$ for which
the weak derivative $f'$ exists and belongs to $L^2(\R)$.
In other words, $H^1 (\R) = W^{1,2}(\R)$ is an $L^2$-Sobolev-space.
It is well known (see \cite[Theorem~7.16]{LeoniSobolevSpaces})
that each $f \in H^1(\R)$ has a representative that is absolutely continuous
on $\R$ and whose classical derivative exists and coincides with the weak derivative $f'$
almost everywhere.

By $\HH^1$ we denote the space of all functions $f\in H^1(\R)$ whose Fourier transform $\wh f$
also belongs to $H^1(\R)$.
Equivalently, a function $f\in L^2(\R)$ is in $\HH^1$ if and only if $f',Xf\in L^2(\R)$,
where $Xf$ represents the function $\R \to \CC, x\mapsto x f(x)$.
The space $\HH^1$ also coincides with the modulation space $M_m^2(\R)$ with the weight
$m(x,\omega) = 1 + \sqrt{x^2 + \omega^2}$; see \cite[Corollary~2.3]{HeilTinaztepe}.

As shown in \cite[Proof of Theorem~1.4]{BLForSubspaces},
the space $\HH^1$ is invariant under symplectic operators,
meaning that $U_B g \in \HH^1$ if $g \in \HH^1$ and $B \in \R^{2 \times 2}$ with $\det B = 1$.

\section{Time-frequency shift invariance: A closer look}
\label{sec:TFInvarianceForFrames}

In this section, we first establish a certain trichotomy concerning the set
of invariant time-frequency shifts.
We then show that one of the three cases of the trichotomy is excluded
if the generator function $g$ is ``sufficiently nice''.

The next theorem establishes the trichotomy:
the invariance set $\Inv(\calG)$ either fills the whole space $\R^2$, or it consists
of equispaced lines that are aligned with the lattice, or it is a refinement of $\La$
(and in particular a lattice itself).
Note that this holds regardless of the regularity of the generator $g$
or the (frame) properties of the Gabor system $(g,\La)$.

\begin{prop}\label{p:trichotomy}
Let $H$ be a closed additive subgroup of $\R^2$ and
suppose that $H \supset \Lambda$ for a non-degenerate lattice $\Lambda \subset \R^2$.
Then there exist $\lambda_1,\lambda_2 \in \Lambda$ satisfying $\La = \Z\cdot\la_1 + \Z\cdot\la_2$
and $m,n\in\N_{\ge 1}$ such that exactly one of the following conditions holds:
\begin{enumerate}
	\item $H = \R^2$.
	\item $H = \R\cdot\la_1 + \Z\cdot\frac{\la_2}{n}$.
	\item $H = \Z\cdot\frac{\la_1}m + \Z\cdot\frac{\la_2}n$.
\end{enumerate}
In particular, if $\Lambda \subset \R^2$ is a lattice and $g \in L^2(\R)$,
then one of the above cases holds for $H = \Inv(\calG(g,\Lambda))$.
\end{prop}
\begin{proof}
By \cite[Theorem~9.11]{HewittRoss}, there are $\alpha,\beta \in \N_0$
and linearly independent vectors $x_1,\dots,x_\alpha,y_1,\dots,y_\beta \in \R^2$
(hence, $\alpha + \beta \leq 2$) such that
\[
  H = \R \, x_1 + \cdots + \R \, x_\alpha + \Z \, y_1 + \cdots + \Z \, y_\beta.
\]
Since $H$ contains the non-degenerate lattice $\Lambda$
(and thus two linearly independent vectors), we must have $\alpha + \beta = 2$.
Hence, there are three cases:
\begin{enumerate}
  \item[(i)]   $(\alpha,\beta) = (2,0)$ and hence $H = \R^2$,
  \item[(ii)]  $(\alpha,\beta) = (1,1)$, so that $H = \R v + \Z w$
        with linearly independent $v,w \in \R^2$,
  \item[(iii)] $(\alpha,\beta) = (0,2)$, so that $H$ is a (non-degenerate) lattice.
\end{enumerate}
Clearly, in Case~(i), Condition~(1) of the statement of the theorem holds.
Let us discuss the case (ii): $H = \R\cdot v + \Z\cdot w$.
Let $\La = \Z\cdot\mu + \Z\cdot\la$ be an arbitrary representation of $\La$.
Since $\La\subset H$, there exist $m,n\in\Z$ and $s,t\in\R$ such that
\[
  [\mu,\la]
  = [sv+mw,\,tv+nw]
  = [v,w]\bmat{s}{t}{m}{n}.
\]
Note that $\mu,\lambda$ are linearly independent, and hence $s n - t m \neq 0$,
so that $d = (sn-tm)^{-1}$ is well-defined.
Furthermore, we see
\[
  [v,w]
  = d \cdot [\mu,\la] \bmat{n}{-t}{-m}{s},
\]
which shows that $v = d (n \mu - m \la)$ and thus $\R \cdot v = \R \cdot \la_1$
with some $\la_1\in\La$.
By rescaling $\lambda_1$, we can ensure that $\tfrac{1}{k} \la_1 \notin \La$
for each $k \in \Z \setminus \{ -1,0,1 \}$.
Note because of $\R \cdot v = \R \cdot \lambda_1$ that $H = \R \cdot \la_1 + \Z \cdot w$.

Now, there exists $\la_2 \in \La$ such that $\La = \Z \cdot \la_1 + \Z \cdot \la_2$.
Indeed, writing $\La = A \Z^2$ with $A = [a_1, a_2] \in \R^{2 \times 2}$ invertible,
there exist $i, j \in \Z$ such that $\la_1 = i a_1 + j a_2$.
The numbers $i,j$ are necessarily coprime, since $\frac{1}{k} \lambda_1 \notin \Lambda$
for $k \in \Z \setminus \{ -1,0,1 \}$.
Hence, by B\'ezout's lemma there exist $k, \ell \in \Z$ such that $i \ell - j k = 1$.
Set $\la_2 = k a_1 + \ell a_2$.
Then $[\la_1, \la_2] \Z^2 = A \sbmat{i}{k}{j}{\ell} \Z^2 = A \Z^2 = \La$.
Since $\lambda_2 \in \La \subset H$, there exist $\sigma \in \R$ and $\nu \in \Z$
such that $\la_2 = \sigma \la_1 + \nu w$.
Then $\nu\neq 0$ and so
\(
  H
  = \R\cdot\la_1 + \Z\cdot(\frac{\la_2}{\nu} - \frac{\sigma}{\nu} \la_1)
  = \R\cdot\la_1 + \Z\cdot\frac{\la_2}{|\nu|}
  .
\)
Hence, Condition~(2) of the statement of the theorem holds.

Assume now that Case~(iii) holds: $H$ is a lattice, i.e., $H = \Z\cdot v + \Z\cdot w$
with linearly independent vectors $v,w\in\R^2$.
Write $\La = \Z\cdot\mu + \Z\cdot\la$ with linearly independent $\lambda,\mu \in \R^2$.
Then, because of $\La \subset H$, there exist $a,b,c,d\in\Z$ such that
\begin{equation}
  [\mu, \la]
  = [a v + c w, b v + d w]
  = [v,w] \bmat{a}{b}{c}{d}
  =: [v,w] \cdot A,
  \label{eq:TrichotomyProofCase3}
\end{equation}
with $A \in \Z^{2 \times 2}$.
Let $A = M D N^{-1}$ be the Smith canonical form of $A$
(see, for instance \mbox{\cite[Theorem~26.2]{macduffee}} or \cite[Theorem~3.8]{jacobson}),
where $M, N, D \in \Z^{2 \times 2}$ with ${\det M = \det N = 1}$ and $D$ is a diagonal matrix.
Note that $A$ (and hence $D$) is invertible; this follows from
\eqref{eq:TrichotomyProofCase3} since $\mu$ and $\lambda$ are linearly independent.
Moreover, note that ${[\mu,\la] N D^{-1} = [v, w] M}$.

Define $\la_1, \la_2 \in \La$ via $[\la_1, \la_2] := [\mu, \la] N$.
Then
\[
  \La
  = [\mu,\la]\Z^2
  = [\mu,\la]N\Z^2
  = [\la_1,\la_2]\Z^2
  = \Z\cdot\la_1 + \Z\cdot\la_2.
\]
Further, writing $D = \diag(m, n)$ with $m, n \in \Z \setminus \{ 0 \}$, we see
\[
  H
  = [v, w] \Z^2
  = [v, w] M \Z^2
  = [\mu, \la] N D^{-1} \Z^2
  = [\la_1, \la_2] D^{-1} \Z^2
  = \Z \cdot \tfrac{\la_1}{|m|} + \Z \cdot \tfrac{\la_2}{|n|} .
\]
This completes the proof of the theorem, since the conditions
(1)--(3) are clearly mutually exclusive.
\end{proof}

The example below shows that Case~(2) in Proposition~\ref{p:trichotomy}
can occur for $H \!=\! \Inv(\calG(g,\Lambda))$ for every lattice $\La$
with density smaller than one---even if $(g,\La)$ is a Riesz sequence.

\begin{ex}\label{ex:case2}
Due to Equation~\eqref{eq:InvariantSetSymplecticTransform}
and Lemma~\ref{lem:SymplecticLatticeSeparation} it suffices to construct an example
for a separable lattice $\La = \alpha\Z\times\beta\Z$ with $\alpha,\beta > 0$, $\alpha\beta > 1$.
For $m \in \Z$, define $E_m := m \alpha + [0, \frac{1}{\beta}]$,
and let $g := \sqrt{\beta} \cdot \one_{E_0}$.
Then
\[
  M_{n\beta}T_{m\alpha}g(x)
  = \sqrt{\beta} \cdot e^{2\pi in\beta x} \cdot\one_{[0,\frac 1\beta]}(x-m\alpha)
  = \sqrt{\beta} \cdot e^{2\pi in\beta x} \cdot\one_{E_m}(x).
\]
Hence, for any $m\in\Z$ the system $(T_{m\alpha}M_{n\beta}g)_{n\in\Z}$ is an orthonormal basis
for the subspace $L^2(E_m)$ of $L^2(\R)$.
Note that, since $\tfrac 1\beta < \alpha$, we have $[0,\frac 1\beta]\subsetneq [0,\alpha]$.
The system $(g,\La)$ is thus an orthonormal basis for $\calG := L^2(E) \subset L^2(\R)$,
where $E = \bigcup_{m\in\Z} E_m$.
Note that $\R \setminus E$ has positive (even infinite) measure, so that $\calG \subsetneq L^2(\R)$.
Moreover, for any $\omega\in\R$ we have
$M_\omega g = \sqrt{\beta} \cdot e^{2\pi i\omega\cdot} \cdot \one_{E_0} \in L^2(E_0) \subset \calG$.
Therefore, Lemma~\ref{lem:StabiliserIsGroup} shows $\{0\} \times \R \subset \frakI(\calG)$,
which can only occur in Case~(2) of Proposition~\ref{p:trichotomy},
since we would have $\calG = L^2(\R)$ in Case~(1), see Lemma~\ref{l:whole_space}.
\end{ex}

Note that the function $g$ in Example~\ref{ex:case2} is not well localized in frequency.
In the remainder of this section, we show that Case~(2) in Proposition~\ref{p:trichotomy}
\emph{cannot} occur if $(g,\Lambda)$ is a frame sequence with a sufficiently nice window $g$.
In this case, the trichotomy from Proposition~\ref{p:trichotomy} becomes a dichotomy.
By $g$ being ``sufficiently nice'' we mean that $g\in\WW(C,\ell^2)$, where
\[
  \WW(C,\ell^2)
  := \big\{
       f\in L^2(\R) : U_B f \in W(C,\ell^2)
       \text{ for all }
       B \in \R^{2\times 2} \text{ with } \det B = 1
     \big\}.
\]
Here, $W(C,\ell^2)$ is the so-called {\em Wiener Amalgam space}
consisting of all {\em continuous} functions $f : \R \to \CC$ satisfying
\[
  \| f \|_{W(C,\ell^2)}
  := \bigg(
     \sum_{k \in \Z} \,\,
       \sup_{x \in [k-1,k+1]}
         |f(x)|^2
     \bigg)^{1/2}
  < \infty .
\]
Recall from Section~\ref{sub:SymplecticOperators} that if $g \in \WW(C,\ell^2)$
and if $B,C \in \R^{2 \times 2}$ satisfy $\det B = \det C = 1$, then there is $\theta_{B,C} \in \CC$
satisfying ${U_C U_B g = \theta_{B,C}\,U_{C B} g \in W(C,\ell^2)}$.
This shows that ${U_B g\in\WW(C,\ell^2)}$ whenever $g \in \WW(C,\ell^2)$ and $\det B = 1$.

Before we prove the announced theorem let us show that the function classes
considered in this paper (namely, $S_0(\R)$ and $\HH^1$) are contained in $\WW(C,\ell^2)$.

\begin{lem}\label{lem:NiceWindowsAreTrichotomyAvoiding}
  We have $S_0(\R)\subset\WW(C,\ell^2)$ and $\HH^1\subset\WW(C,\ell^2)$.
\end{lem}

\begin{proof}
  If $g \in S_0(\R)$, then \cite[Proposition~12.1.3]{GroechenigTFFoundations}
  shows that $U_B g \in S_0(\R)$ for each $B \in \R^{2 \times 2}$ with $\det B = 1$.
  Similarly, if $g \in \HH^1$, then \cite[Proof of Theorem~1.4]{BLForSubspaces}
  shows that $U_B g \in \HH^1$ for each $B \in \R^{2 \times 2}$ with $\det B = 1$.
  Therefore, it suffices to show that $S_0(\R) \subset W(C, \ell^2)$
  and $\HH^1 \subset W(C,\ell^2)$.

  First, if $f \in S_0(\R)$, then \cite[Proposition~12.1.4]{GroechenigTFFoundations}
  shows that $\widehat{f}\in L^1(\R)$.
  By Fourier inversion, this implies that $f$ has a continuous representative.
  Furthermore, by \cite[Proposition~12.1.4]{GroechenigTFFoundations}
  we have $f\in W(L^\infty, \ell^1)$.
  Combined with the embedding ${\ell^1(\Z) \hookrightarrow \ell^2(\Z)}$,
  this easily implies $f \in W(L^\infty,\ell^2)$ and thus $f \in W(C,\ell^2)$.

  Next, if $f \in \HH^1 \subset H^1 = W^{1,2}(\R)$,
  then \cite[Theorem~7.16]{LeoniSobolevSpaces} shows (after changing $f$ on a null-set)
  that $f$ is absolutely continuous, and hence continuous,
  and satisfies $f(x) - f(y) = \int_y^x f'(t) \, d t$ for all $y < x$,
  where $f' \in L^2(\R)$ is the weak derivative of $f$.
  Now, note that if $n \in \Z$ and $x,y \in [n-1, n+1]$, then
  \begin{align*}
  |f(x)|
  & \leq |f(y)| + \int_{\min \{ x,y \}}^{\max\{ x,y \}} |f'(t)| \, dt
      \leq |f(y)| + \int_{n-1}^{n+1} |f'(t)| \, d t \\
  & \leq |f(y)| + \sqrt{2} \left( \int_{n-1}^{n+1} |f'(t)|^2 \, dt \right)^{1/2} ,
  \end{align*}
  and hence $|f(x)|^2 \leq 2 |f(y)|^2 + 4 \int_{n-1}^{n+1} |f'(t)|^2 \, d t$.
  Integrating this over $y \in [n-1,n+1]$ gives
  \[
      2|f(x)|^2\leq 2 \int_{n-1}^{n+1} |f(y)|^2 \, d y + 8 \int_{n-1}^{n+1} |f'(t)|^2 \, d t ,
  \]
  for all $x\in [n-1,n+1]$, which finally implies
  \[
      \| f \|_{W(C,\ell^2)}^2
      \leq \sum_{n \in \Z}
           \left(
             \int_{n-1}^{n+1} |f(y)|^2 \, d y
             + 4 \int_{n-1}^{n+1} |f'(t)|^2 \, d t
           \right)
      \lesssim \| f \|_{L^2}^2 + \| f' \|_{L^2}^2
      < \infty ,
  \]
  and hence $f \in W(C,\ell^2)$.
\end{proof}

Our next result shows that Case~(2) in Proposition~\ref{p:trichotomy} cannot occur
if $(g,\Lambda)$ is a frame sequence with generator $g \in \WW(C,\ell^2)\setminus\{0\}$.

\begin{thm}\label{thm:DichotomyForNiceWindows}
Let $g \in \WW(C,\ell^2) \backslash \{0\}$ and let $\Lambda \subset \R^2$ be a lattice
such that $(g,\Lambda)$ is a frame for $\calG = \calG(g, \Lambda)$.
Then either $\frakI(\calG) = \R^2$ or there exist $\la_1,\la_2\in\La$
and $m,n\in\N_{\ge 1}$ such that
\begin{equation}\label{e:refinement}
  \La = \Z\cdot\la_1 + \Z\cdot\la_2
  \qquad\text{and}\qquad
  \frakI(\calG) = \Z\cdot\tfrac{\la_1}m + \Z\cdot\tfrac{\la_2}n.
\end{equation}
\end{thm}

\begin{proof}
  Let us assume $\Inv(\calG) \subsetneq \R^2$.
  Writing $H = \Inv(\calG)$, the two possibilities in \eqref{e:refinement}
  represent the cases (1) and (3) from the trichotomy in Proposition~\ref{p:trichotomy}.
  It is thus enough to show that Case~(2) from that theorem cannot occur.
  Therefore, we assume towards a contradiction that Case~(2) holds,
  i.e., there are $\la_1,\la_2\in\La$ and $n\in\N_{\ge 1}$ such that
  $\La = \Z\cdot\la_1 + \Z\cdot\la_2$ and $\frakI(\calG) = \Z\cdot\frac{\la_1}{n} + \R\cdot\la_2$.

  {\bf Step 1.} We first derive a contradiction for the case $\la_1 = (\alpha,0)^\topp$
  and $\la_2 = (0,\beta)^\topp$ with some $\alpha,\beta > 0$.
  Then $\Lambda = \alpha\Z\times\beta\Z$, and $\{ 0 \}\times\R\subset\Inv(\calG)$.
  For $f\in\calG$ we thus have $M_\omega f\in\calG$ for all $\omega\in\R$.
  By \cite[Theorem~9.17]{RudinRealAndComplexAnalysis}
  (applied to the translation invariant space $\mathcal{F}^{-1} \calG$,
  with $\calF$ denoting the Fourier transform),
  there exists a Borel measurable set $E \subset \R$ such that $\calG = L^2(E)$,
  where we consider $L^2(E)$ as a closed subspace of $L^2(\R)$, in the sense that
  ${L^2 (E) = \{ f \in L^2(\R) \colon f = 0 \text{ a.e.~on } \R \backslash E \}}$.

  Our goal is to show that $E = \R$, up to null-sets.
  This will imply ${\calG = L^2(E) = L^2(\R)}$ and hence $\Inv(\calG) = \R^2$,
  providing the desired contradiction.
  Towards proving $E = \R$, let us consider for given $f \in L^2(\R)$
  the continuous function $\Gamma_f : \R \to \R$ defined by
  \[
    \Gamma_f(\omega)
    := \<S M_\omega f, M_\omega f\>,
    \quad\omega\in\R,
  \]
  where $S : L^2(\R)\to\calG$ denotes the frame operator of $(g,\Lambda)$.
  By \cite[Proposition~7.1.1]{GroechenigTFFoundations}, the operator $S$ has the
  \emph{Walnut representation}
  \[
    \langle S f, h \rangle
    = \beta^{-1} \sum_{n\in\Z}
                   \langle G_n \cdot T_{\frac n\beta} f, h \rangle
    \qquad \forall \, f,h \in L^\infty(\R) \text{ with compact support},
  \]
  where only finitely many terms of the sum do not vanish, and where
  \[
    G_n(x)
    := \sum_{m\in\Z} g(x - m\alpha) \cdot \ol{g(x - \tfrac{n}{\beta} - m\alpha)},
    \qquad x \in \R, n \in \Z .
  \]
  The fact that $g \in W(C,\ell^2)$ easily implies that the series
  defining $G_n$ converges locally uniformly, and that the $G_n$ are continuous functions.
  Since $G_n$ is also $\alpha$-periodic, this means that each $G_n$ is bounded.

  Now, since multiplication with $G_n$ commutes with the modulation $M_\omega$, using the identity
  ${T_{n/\beta} M_\omega = e^{-2\pi i \frac{n}{\beta} \omega} M_\omega T_{n/\beta}}$, we get
  \begin{equation}
    \Gamma_f(\omega)
    = \beta^{-1} \sum_{n\in\Z}
                   e^{-2\pi i\frac n\beta \omega}
                   \<G_n\cdot T_{\frac n\beta} f,f\>
    \qquad \forall \, f \in L^\infty (\R) \text{ with compact support},
    \label{eq:SpecialWalnutFormula}
  \end{equation}
	where there are only finitely many $n \in \Z$ (depending only on $f$, but not on the choice
  of $\omega$) for which $\langle G_n \cdot T_{\frac{n}{^\beta}} f, f \rangle \neq 0$.

  As $(g,\Lambda)$ is a frame for $\calG$ and $M_\omega f\in\calG$
  for all $\omega \in \R$ and $f \in \calG$, there exists $A > 0$
  such that $\Gamma_f(\omega) =  \< S M_\omega f, M_\omega f\> \ge A \|f\|_{L^2}^2$
  for all $f \in \calG$.
  Let us write $L^\infty_c(E)$ for the set of all compactly supported $f \in L^\infty(\R)$
  which satisfy $f = 0$ on $\R \backslash E$, and note that $L^\infty_c(E) \subset L^2(E) = \calG$.
  For $f \in L^\infty_c(E)$, integrate the estimate $\Gamma_f(\omega) \geq A \|f\|_{L^2}^2$
  over $[0, \beta]$ and apply Equation~\eqref{eq:SpecialWalnutFormula} to see
  \[
    \beta \, A \, \|f\|_{L^2}^2
    \,\leq\, \beta^{-1} \sum_{n\in\Z}
                          \<G_n\cdot T_{\frac n\beta}f,f\>
                          \int_0^\beta
                            e^{-2\pi i\frac n\beta \omega}
                          \,d\omega
    = \<G_0f,f\> = \<hf,f\>,
  \]
  where $h := G_0 = \sum_{m\in\Z}|T_{m\alpha}g|^2$.
  We have thus shown
  \[
    \int_E \big( h(x) - \beta A \big) \cdot |f(x)|^2 \, dx \,\geq\, 0
    \qquad \forall \, f \in L^\infty_c(E).
  \]
  Using standard arguments, this implies that $h(x) \geq \beta A$ for almost all $x \in E$.

  Since $T_{m\alpha} \, g \in \calG = L^2(E)$ and thus $T_{m\alpha} \, g(x) = 0$
  for almost all $x \in \R \backslash E$ and arbitrary $m \in \Z$, it follows that $h(x) = 0$
  for almost all $x \in \R \backslash E$.
  Recall from above that $h(x) \geq \beta A$ for almost all $x \in E$;
  thus, $h(x) \in \{0\} \cup [\beta A, \infty)$ almost everywhere.
  Also recall from above that $h = G_0$ is continuous.
  Hence, the \emph{open} set $h^{-1}((0, \beta A))$ has measure zero and is thus empty;
  that is, $h(x) \in \{0\} \cup [\beta A, \infty)$ for \emph{all} $x \in \R$.
  By the intermediate value theorem, this implies that
  $h(x) \geq \beta A$ for all $x \in \R$ (since $h \geq |g|^2$ and $g \not\equiv 0$) and thus,
  indeed, $E = \R$ (up to null-sets), since $h(x) = 0$ a.e.~on $\R \backslash E$.

  \medskip{}

  {\bf Step 2.}
	Let $\La$ be a general lattice.
  Recall that $\La = \Z \la_1 + \Z \la_2$
  and ${\frakI(\calG) = \Z \frac{\la_1}{n} + \R \la_2}$.
	By Lemma~\ref{lem:SymplecticLatticeSeparation} there exists $B \in \R^{2\times 2}$
  with $\det B = 1$ such that $B [\la_1, \la_2] = \diag(\alpha, \beta)$
  for certain $\alpha, \beta \in \R \setminus \{0\}$.
  We thus obtain $\La_B = B\La = B [\la_1, \la_2] \Z^2 = |\alpha| \Z \times |\beta| \Z$ and
  \[
    \Inv(\calG_B)
    = B \Inv(\calG)
    = B [\la_1, \la_2] \diag(\tfrac{1}{n}, 1)(\Z \times \R)
    = \diag(\tfrac{|\alpha|}{n}, |\beta|) (\Z \times \R)
    = \tfrac{|\alpha|}{n} \Z \times \R;
  \]
  see \eqref{eq:InvariantSetSymplecticTransform}.
  In particular, $\{0\}\times\R\subset\Inv(\calG_B)$.
  Hence, since $g_B = U_Bg\in\WW(C,\ell^2)$ and $(g_B,\La_B)$ is a frame for
  $\calG_B = U_B \calG \subsetneq L^2(\R)$ (cf.\ Subsection \ref{sub:SymplecticOperators}),
  we are in the situation of Step~1, which we proved to be impossible.
\end{proof}

By combining Theorem~\ref{thm:DichotomyForNiceWindows}
and Lemma~\ref{lem:NiceWindowsAreTrichotomyAvoiding}, we obtain the following corollary.

\begin{cor}\label{cor:AdditionalTFShiftControl}
  Let $g \in S_0(\R) \setminus \{ 0 \}$ or $g \in \HH^1 \setminus \{ 0 \}$
  and let $\Lambda\subset\R^2$ be a lattice
  such that $(g,\Lambda)$ is a Riesz basis for $\calG := \calG(g,\Lambda)$.
  Then $\Inv(\calG)$ is a refinement of $\La$ as in \eqref{e:refinement}.
\end{cor}

\begin{proof}
  By the Balian-Low theorem \cite[Theorem~2.3]{DaubechiesBalianLow}
  and the Amalgam Balian-Low theorem \cite[Theorem~3.2]{BenedettoDifferentiationAndBLT},
  it is not possible that $\calG = L^2(\R)$.
  Therefore, Lemma~\ref{l:whole_space} implies $\Inv(\calG)\neq\R^2$.
  The rest follows from Lemma~\ref{lem:NiceWindowsAreTrichotomyAvoiding}
  and Theorem~\ref{thm:DichotomyForNiceWindows}.
\end{proof}

\section{Time-frequency shift invariance: Duality}
\label{sec:TFInvariantAndFrameSequences}

Let us consider a Gabor Riesz sequence $(g,\Lambda)$ with $g\in\WW(C,\ell^2)$
as in the previous section, and assume that $\calG := \calG(g,\Lambda) \subsetneq L^2(\R)$,
but that there exists an additional time-frequency shift, meaning $\Inv(\calG) \neq \Lambda$.
In view of Theorem~\ref{thm:GeneralizedBLConjecture_Intro} it is our goal
to show that this is impossible, at least if $g \in S_0$.
To make the situation more accessible, we first reduce to the case
where $\Lambda = \alpha \Z \times \beta \Z$ is separable,
and where the additional time-frequency shift is of the form $(\frac{\alpha}{\nu}, 0)^\topp$
for some $\nu \in \N_{\geq 2}$, meaning that $T_{\alpha/\nu} \, g \in \calG$.
After that, we provide a characterization of this simplified condition
in terms of the adjoint Gabor system.
It is this characterization that we will use to prove our main result,
Theorem~\ref{thm:GeneralizedBLConjecture_Intro}, in the next section.

\begin{lem}\label{l:ReductionToSeparableLattice}
  Let $g \in \WW(C,\ell^2) \backslash \{0\}$ and let $\Lambda \subset \R^2$ be a lattice
  such that $(g,\Lambda)$ is a frame for $\calG := \calG(g, \Lambda)$.
  If $\calG \neq L^2(\R)$ and $\Inv(\calG) \neq \La$, there exist a matrix $B \in \R^{2\times 2}$
  with $\det B = 1$ and $\alpha, \beta > 0$ such that $\La_B = \alpha \Z \times \beta \Z$
  and $(\tfrac{\alpha}{\nu},0)^\topp  \in \Inv(\calG_B)$ for some $\nu \in \N$,
  $\nu \ge 2$ \braces{i.e., $T_{\frac{\alpha}{\nu}} g_B \in \calG_B$}.
\end{lem}

\begin{proof}
  Due to Theorem~\ref{thm:DichotomyForNiceWindows} and Lemma~\ref{l:whole_space}, we have
  $\La \!=\! [\la_1,\la_2]\,\Z^2$ and ${\Inv(\calG) \!=\! [\frac{\la_1}m,\frac{\la_2}n]\,\Z^2}$
  for suitable vectors $\la_1, \la_2 \in \R^2$ and $m, n \in \N \backslash \{0\}$.
  We may safely assume that $m \neq 1$.
  Indeed, since $\Inv(\calG)\neq\La$, we have $(m,n)\neq (1,1)$.
  If $m=1$, then with $J = \sbmat 01{-1}0$ also $\La = [\la_1,\la_2]J\Z^2 = [-\la_2,\la_1]\Z^2$
  and $\Inv(\calG) = [\frac{-\la_2}n,\frac{\la_1}m]\Z^2$.

  Now, by Lemma~\ref{lem:SymplecticLatticeSeparation} there exists a matrix $B\in\R^{2\times 2}$
  with $\det B=1$ such that $B[\la_1,\la_2] = \diag(\alpha,\beta)$,
  where $\alpha,\beta\in\R\backslash\{0\}$.
  Hence, $\La_B = B \La = |\alpha| \Z \times |\beta| \Z$ and
  \[
    \Inv(\calG_B)
    = B\Inv(\calG)
    = B[\la_1,\la_2]\diag(\tfrac 1m,\tfrac 1n)\Z^2
    = \tfrac{|\alpha|}{m} \Z \times \tfrac{|\beta|}{n} \Z;
  \]
  see \eqref{eq:InvariantSetSymplecticTransform}.
  In particular, $(\tfrac{|\alpha|}{m}, 0)^\topp \in \Inv(\calG_B)$ and $m \geq 2$.
\end{proof}

In what follows, fix $g \in L^2(\R)$, $\alpha, \beta > 0$, $\Lambda = \alpha \Z \times \beta \Z$,
and $\nu \in \N_{\geq 2}$, and assume that $(g,\Lambda)$ is a frame for $\calG = \calG(g,\Lambda)$.
The adjoint system ${\calF := \{ T_{k/\beta} M_{\ell/\alpha} \, g : k, \ell \in \Z\}}$
is then a frame for its closed linear span $\calK$ by \cite[Theorem~2.2 (c)]{RonShenDuality}.
Note that $\calK = L^2(\R)$ if and only if $(g,\La)$ is a Riesz sequence
(cf.~\cite[Thm.\ 2.2 (e)]{RonShenDuality} or \cite[Theorem~7.4.3]{GroechenigTFFoundations}).

It is a natural question to ask what the existence of an additional time-frequency shift
of the form $T_{\frac\alpha\nu}g\in\calG$ means for the adjoint system $\calF$.
To describe this, we set
\[
  \calF_s
  := \big\{
       T_{\frac k\beta} M_{\frac{\ell\nu}\alpha} M_{\frac s\alpha} g
       \colon
       k,\ell\in\Z
     \big\},
  \qquad s=0,\ldots,\nu-1.
\]
Again by \cite[Theorem~2.2~(c)]{RonShenDuality}, $\calF_0$ is a frame sequence if and only if
the system ${(g,\tfrac\alpha\nu\Z\times\beta\Z)}$ is a frame sequence.
In this case, each $\calF_s$ is a frame sequence because $M_{s/\alpha}\calF_0$ is,
and multiplying the vectors of a frame sequence by unimodular constants results in a frame sequence.
We set $\calL_s := \ol{\vphantom{t} \linspan}\,\calF_s$ for $s \in \{ 0, \ldots, \nu-1 \}$.
Note that
\[
\calK = \calL_0 + \cdots + \calL_{\nu-1}.
\]
Indeed, the inclusion ``$\supset$'' is trivial.
Conversely, since $\calF$ is a frame sequence, each ${f \in \calK}$ satisfies
$f = \sum_{k,\ell \in \Z} c_{k,\ell} \, T_{k/\beta} M_{\ell/\alpha} \, g$
with a suitable sequence ${c = (c_{k,\ell})_{k,\ell \in \Z} \! \in \! \ell^2(\Z^2)}$.
Since $\calF$ is a Bessel sequence,
\({
  f_s
  := \sum_{k,\ell \in \Z}
       c_{k, \ell \nu + s} T_{\frac{k}{\beta}} M_{\frac{\ell \nu + s}{\alpha}} g
  \in \calL_s
}\)
is well-defined for $s \in \{ 0, \dots, \nu-1 \}$, and
$f = f_0 + \dots + f_{\nu - 1} \in \calL_0 + \dots + \calL_{\nu - 1}$.
Finally, it is clear that $\calL_s = M_{\frac s\alpha}\calL_0$.

In the sequel, the symbol $\directSum$ denotes the {\em direct} (not necessarily orthogonal)
sum of subspaces, whereas $\oplus$ is used to denote an orthogonal sum.
The next theorem characterizes the existence of an additional time-frequency shift
for $\calG$ in terms of properties of the adjoint system $\calF$.

\begin{thm}\label{thm:AdditionalTFShiftAdjointCharacterization}
  Let $g \in L^2(\R)$ and $\alpha, \beta > 0$, and assume that $(g, \alpha \Z \times \beta \Z)$
  is a frame sequence with canonical dual window $\gamma \in \calG$, where
  $\calG = \calG(g, \alpha \Z \times \beta \Z)$.
  Let $\nu \in \N_{\geq 2}$, and define the systems $\calF_s$ and the spaces $\calK,\calL_s$
  as above, and set $S_{\gamma,g} := S_{\frac{1}{\beta}\Z \times \frac{\nu}{\alpha}\Z, \gamma, g}$,
  with notation as in Equation~\eqref{eq:FrameOperatorDefinition}.
  Then the following statements are equivalent:
  \begin{enumerate}
    \item[{\rm (i)}]   $T_{\frac\alpha\nu}g\in\calG$.
                       \vspace{0.1cm}

    \item[{\rm (ii)}]  $(\alpha\beta)^{-1}S_{\gamma,g}M_{\frac s\alpha}g = \delta_{s,0}\cdot g$
                       for $s=0,\ldots,\nu-1$.
                       \vspace{0.1cm}

    \item[{\rm (iii)}] $\calK = \calL_0 \directSum \cdots \directSum \calL_{\nu-1}$.
                       \vspace{0.1cm}

    \item[{\rm (iv)}] $\big\<T_{\frac k\beta}M_{\frac{\ell}\alpha}\gamma,g\big\> = 0$
                      for all $k \in \Z$ and all $\ell \in \Z \backslash \nu\Z$.
  \end{enumerate}
  If one of {\rm (i)--(iv)} holds, then for each $s = 0,\dots,\nu - 1$
  the system $\calF_s$ is a frame for $\calL_s$ and the operator
  \(
    P_s := (\alpha\beta)^{-1} M_{s/\alpha} S_{\gamma,g} M_{-s/\alpha}
  \)
  is the \braces{possibly non-orthogonal} projection onto $\calL_s$
  with respect to the decomposition
  \(
    L^2(\R)
    = (\calL_0 \directSum \cdots \directSum \calL_{\nu-1}) \oplus \calK^\perp.
  \)
\end{thm}

\begin{proof}
First, note by Ron-Shen duality (see \cite[Theorem~2.2(c)]{RonShenDuality}) that
$\calF$ is a frame sequence.
We will frequently use the following fact (see \cite[Theorem~2.3]{RonShenDuality}):
\begin{equation}\label{eq:CanonicalDualOfAdjointSystem}
  \text{
    $(\alpha\beta)^{-1}\gamma$ is the canonical dual window
    of $\calF = \big\{ T_{\frac{k}{\beta}} M_{\frac{\ell}{\alpha}} g : k,\ell\in\Z \big\}$
  };
\end{equation}
in particular, $\gamma \in \overline{\vphantom{t} \linspan} \, \calF = \calK$.

For the rest of the proof we set
\(
  P
  := (\alpha\beta)^{-1} S_{\gamma,g}
  = (\alpha \beta)^{-1} S_{\frac{1}{\beta}\Z \times \frac{\nu}{\alpha} \Z, \gamma, g}
  .
\)
It is well known (see for instance \cite[Equation~(5.25)]{GroechenigTFFoundations}) that
\begin{equation}\label{eq:SpecialFrameOperatorCommutationRelation1}
  P T_{\frac{k}{\beta}} M_{\frac{\ell\nu}{\alpha}}
  = T_{\frac{k}{\beta}} M_{\frac{\ell\nu}{\alpha}} P\qquad\text{for all }k,\ell\in\Z.
\end{equation}
Moreover, Equation~\eqref{e:3bessel} applied to the lattice
$\frac{1}{\beta} \Z \times \frac{\nu}{\alpha} \Z$ shows
for $f \in \calB_{\frac{1}{\beta}\Z \times \frac{\nu}{\alpha} \Z}$ that
\begin{equation}
  Pf = \sum_{m,n \in \Z}
         c_{m,n} \cdot T_{\frac{m\alpha}{\nu}} M_{n\beta} f
  \qquad\text{with}\qquad
  c_{m,n} = \tfrac{1}{\nu} \< g, T_{\frac{m\alpha}{\nu}} M_{n\beta} \gamma \> .
  \label{eq:SpecialProjectionRepresentation}
\end{equation}

Let us denote the orthogonal projection onto the subspace
$\calK = \overline{\vphantom{t} \linspan} \, \calF$ by $P_\calK$.
Note that Equation~\eqref{eq:CanonicalDualOfAdjointSystem} implies
\(
  S_{\frac{1}{\beta}\Z \times \frac{1}{\alpha}\Z, (\alpha \beta)^{-1} \gamma, g}|_{\calK}
  = \mathrm{id}_{\calK}
\)
and
\(
  S_{\frac{1}{\beta} \Z \times \frac{1}{\alpha}\Z, (\alpha \beta)^{-1} \gamma, g}|_{\calK^\perp}
  \equiv 0,
\)
so that $P_{\calK} = S_{\frac{1}{\beta}\Z \times \frac{1}{\alpha} \Z, (\alpha \beta)^{-1} \gamma, g}$.
Similarly, the orthogonal projection $P_{\calG}$ onto $\calG$ satisfies
$P_{\calG} = S_{\alpha\Z\times\beta\Z, \gamma, g}$.
Next, using \eqref{eq:SpecialProjectionRepresentation} and the elementary identity
${\sum_{s=0}^{\nu-1} e^{2\pi i \frac{m s}{\nu}} \!= \nu \!\cdot\! \Indicator_{\nu \Z} (m)}$,
we obtain
\begin{align}
  \begin{split}\label{eq:ProjectionDecomposition}
    \sum_{s=0}^{\nu-1}
      M_{\frac{s}{\alpha}}
      PM_{-\frac{s}{\alpha}}f
    &= \sum_{s=0}^{\nu-1} \,\,
         \sum_{m,n \in \Z}
           c_{m,n}\cdot
           M_{\frac{s}{\alpha}}
           T_{\frac{m\alpha}{\nu}}
           M_{n\beta}
           M_{-\frac{s}{\alpha}}f \\
    &= \sum_{m,n \in \Z}
         c_{m,n}
         \bigg(
           \sum_{s=0}^{\nu-1}
             e^{2\pi i\frac{ms}{\nu}}
         \bigg)
         T_{\frac{m\alpha}{\nu}}
         M_{n\beta}f \\
    &= \nu \sum_{m,n \in \Z}
             c_{\nu m, n}
             T_{m\alpha}
             M_{n\beta}f
     = \sum_{m,n \in \Z}
         \< g, T_{m\alpha} M_{n\beta} \gamma \>
         \cdot T_{m\alpha} M_{n\beta}f \\
    ({\scriptstyle{\text{Equation } \eqref{e:3bessel}}})
    &= (\alpha \beta)^{-1}
       S_{\frac{1}{\beta}\Z \times \frac{1}{\alpha}\Z, \gamma, g} f
     = S_{\frac{1}{\beta} \Z \times \frac{1}{\alpha} \Z, (\alpha\beta)^{-1}\gamma, g} f
     = P_\calK f,
  \end{split}
\end{align}
for all $f \in \calB_{\frac{1}{\beta} \Z \times \frac{1}{\alpha} \Z}$
and hence for all $f\in L^2(\R)$ by density.
Here, we used that if $f \in \calB_{\frac{1}{\beta} \Z \times \frac{1}{\alpha} \Z}$, then
\({
  M_{-\frac{s}{\alpha}} f
  \in \calB_{\frac{1}{\beta} \Z \times \frac{1}{\alpha} \Z}
  \subset \calB_{\frac{1}{\beta} \Z \times \frac{\nu}{\alpha} \Z} .
}\)
Next, for $s = 0, \ldots, \nu-1$, we see
by another application of Equation~\eqref{eq:SpecialProjectionRepresentation} that
\begin{align}
  \begin{split}\label{eq:SpecialFrameOperatorConjugation}
    M_{-\frac{s}{\alpha}} P M_{\frac{s}{\alpha}} g
    &= \frac{1}{\nu}
       \sum_{m,n \in \Z}
         \< g, T_{\frac{m\alpha}{\nu}} M_{n\beta} \gamma\>
         \cdot M_{-\frac{s}{\alpha}}
               T_{\frac{m\alpha}{\nu}}
               M_{n\beta}
               M_{\frac{s}{\alpha}} g\\
    &= \frac{1}{\nu}
       \sum_{m,n \in \Z} \,
         \sum_{r=0}^{\nu-1}
           \big\< g, T_{\frac{\nu m-r}{\nu} \alpha} M_{n\beta} \gamma\big\>\cdot
           M_{-\frac{s}{\alpha}}
           T_{\frac{\nu m-r}{\nu} \alpha}
           M_{n\beta}
           M_{\frac{s}{\alpha}} g \\
    &= \frac{1}{\nu}
       \sum_{r=0}^{\nu-1}
         e^{2\pi i\frac{sr}{\nu}}\cdot
         T_{-\frac{r\alpha}{\nu}}
         \sum_{m,n \in \Z}
           \big\< T_{\frac{r\alpha}{\nu}}g, T_{m\alpha} M_{n\beta} \gamma\big\>\cdot
           T_{m\alpha} M_{n\beta} \, g \\
    &= \frac{1}{\nu}
       \sum_{r=0}^{\nu-1}
         e^{2\pi i\frac{sr}{\nu}}
         \cdot T_{-\frac{r\alpha}{\nu}} P_\calG T_{\frac{r\alpha}{\nu}} g,
  \end{split}
\end{align}
where $P_\calG$ is the orthogonal projection onto $\calG$.
Equation~\eqref{eq:SpecialFrameOperatorConjugation} shows that the vectors
\[
  v = \big(
        M_{-\frac{s}{\alpha}} P M_{\frac{s}{\alpha}} g
      \big)_{s=0}^{\nu-1}
  \qquad\text{and}\qquad
  u = \big(
        T_{-\frac{r\alpha}{\nu}} P_\calG T_{\frac{r\alpha}{\nu}} g
      \big)_{r=0}^{\nu-1}
\]
in $\big( L^2(\R) \big)^\nu$ satisfy $F_\omega u = \sqrt\nu\cdot v$,
where $F_\omega$ is the DFT-matrix ${F_\omega = \nu^{-1/2} \, \bigl(\omega^{sr}\bigr)_{s,r=0}^{\nu-1}}$
with $\omega = e^{2\pi i/\nu}$.

With this preparation, we now prove the equivalence of the statements (i)--(iv).

\textbf{(i)$\boldsymbol{\Leftrightarrow}$(ii):}
If $T_{\frac{\alpha}{\nu}} g \in \calG$, then Lemma~\ref{lem:StabiliserIsGroup} shows that
$T_{\frac{r\alpha}{\nu}} g \in \calG$ for all $r \in \Z$, so that
$T_{-\frac{r \alpha}{\nu}} P_{\calG} T_{\frac{r \alpha}{\nu}} g = g$ for all $r \in \Z$.
Since $\frac{1}{\nu} \sum_{r=0}^{\nu-1} e^{2 \pi i \frac{s r}{\nu}} = \delta_{s,0}$
for $s \in \{ 0,\dots,\nu-1 \}$, Property~(ii) then
follows from \eqref{eq:SpecialFrameOperatorConjugation}.
Conversely, if (ii) holds, then $v = (g, 0, \ldots, 0)^\topp $,
which implies that ${u = \sqrt{\nu} \cdot F_\omega^*v = (g, g, \ldots, g)^\topp }$.
In particular, $T_{-\alpha/\nu} P_\calG T_{\alpha/\nu} g = g$,
i.e., $T_{\alpha/\nu} \, g \in \calG$.

\smallskip{}

\textbf{(ii)$\boldsymbol{\Rightarrow}$(iii):}
Since $Pg = g$, it is a consequence
of \eqref{eq:SpecialFrameOperatorCommutationRelation1} that $P|_{\calL_0} = I|_{\calL_0}$.
Furthermore, for $s \in \{ 1, \dots, \nu - 1 \}$ and $k,\ell \in \Z$,
Equation~\eqref{eq:SpecialFrameOperatorCommutationRelation1} implies
\[
  P T_{\frac{k}{\beta}} M_{\frac{\ell\nu}{\alpha}} M_{\frac{s}{\alpha}} g
  = T_{\frac{k}{\beta}} M_{\frac{\ell\nu}{\alpha}} P M_{\frac{s}{\alpha}} g
  = 0,
\]
which shows $P|_{\calL_s} = 0$.
By using these observations and by noting $\calL_r = M_{r/\alpha} \calL_0$,
we see for $r,s \in \{ 0, \ldots, \nu-1 \}$ that
$P_r |_{\calL_r} = M_{r/\alpha}PM_{-r/\alpha}|_{\calL_r} = I|_{\calL_r}$
and furthermore $P_r |_{\calL_s} = M_{r/\alpha}PM_{-r/\alpha}|_{\calL_s} = 0$ for $s\neq r$.
Hence, the sum $\calK = \calL_0 \directSum \cdots \directSum \calL_{\nu-1}$ is direct,
and $P_s|_{\calK} = M_{s/\alpha}PM_{-s/\alpha}|_\calK$ is the projection onto $\calL_s$
with respect to this decomposition.
Finally, since $\gamma \in \calK$ and since $\calK$ is invariant under $T_{k/\beta} M_{\ell/\alpha}$,
it follows by definition of
\(
  P_s
  = (\alpha \beta)^{-1}
    M_{s/\alpha} \,
    S_{\frac{1}{\beta} \Z \times \frac{\nu}{\alpha} \Z, \gamma, g} \,
    M_{-s/\alpha}
\)
that $P_s |_{\calK^{\perp}} = 0$.
Therefore, $P_s$ is the projection onto $\calL_s$ with respect to the decomposition
$L^2(\R) = (\calL_0 \directSum \cdots \directSum \calL_{\nu - 1}) \oplus \calK^{\perp}$.

Finally, we show that $\calF_s$ is a frame for $\calL_s$,
where it clearly suffices to show this for $s = 0$.
Since $\calF_0$ is a Bessel sequence, \cite[Corollary~5.5.2]{ChristensenBook} shows
that we only need to prove that the synthesis operator
\[
  D : \quad
  \ell^2(\Z^2) \to L^2(\R),\qquad
  (c_{k,\ell})_{k,\ell \in \Z} \mapsto \sum_{k,\ell \in \Z}
                                         c_{k,\ell} \,
                                         T_{\frac{k}{\beta}} \,
                                         M_{\frac{\ell\nu}{\alpha}} \, g
\]
has closed range $\ran D = \calL_0$.
By definition of $\calL_0 = \overline{\vphantom{t} \linspan} \, \calF_0$,
we see ${\ran D \subset \calL_0}$.
Conversely, if $f \in \calL_0$, then
$f = P f = (\alpha \beta)^{-1} S_{\gamma,g} f = (\alpha \beta)^{-1} D c \in \ran D$
for the sequence $c = (c_{k,\ell})_{k,\ell \in \Z} \in \ell^2(\Z^2)$ given by
$c_{k,\ell} = \langle f, T_{k/\beta} M_{\ell \nu / \alpha} \gamma \rangle$.

\smallskip{}

\textbf{(iii)$\boldsymbol{\Rightarrow}$(ii):}
Since $P = (\alpha\beta)^{-1}S_{\gamma,g}$, we see by definition of $S_{\gamma,g}$
that $\ran P\subset\calL_0$.
Hence, $Pg - g \in \calL_0$.
On the other hand, again as a consequence of $\ran P \subset \calL_0$ we see that
$M_{s/\alpha} P M_{-s/\alpha} \, g \in \calL_s$, so that Equation~\eqref{eq:ProjectionDecomposition}
implies
\vspace{-0.2cm}
\begin{equation}
  \calL_0
  \ni P \, g - g
  = P \, g - P_{\calK} \, g
  = - \sum_{s=1}^{\nu - 1} M_{s/\alpha} \, P \, M_{-s/\alpha} \, g
  \in \calL_1 + \dots + \calL_{\nu - 1},
  \label{eq:DirectSumApplication}
\end{equation}
and thus $P g = g$ since the sum $\calL_0 + \dots + \calL_{\nu - 1}$ is direct.
Similarly, for any ${s \!\in\! \{1,\ldots,\nu \!-\! 1\}}$ we get because of
$\ran P \subset \calL_0$ that $M_{s/\alpha} P M_{-s/\alpha} \, g \in \calL_s$;
but this implies as in Equation~\eqref{eq:DirectSumApplication}
that
\[
  \calL_s
  \ni M_{\frac{s}{\alpha}} P M_{-\frac{s}{\alpha}} \, g
  = P_{\calK} g - \sum_{r \neq s} M_{\frac{r}{\alpha}} P M_{-\frac{r}{\alpha}} \, g
  \in \calL_0 + \linspan \{ \calL_r \colon r \neq s \}
  = \linspan \{ \calL_r \colon r \neq s \} .
\]
Again, since $\calL_0 + \dots + \calL_{\nu - 1}$ is a direct sum, this implies
$P M_{-\frac{s}{\alpha}} \, g = 0$ for ${s = 1,\ldots,\nu-1}$.
Since $P$ commutes with $M_{\pm\nu/\alpha}$ (see \eqref{eq:SpecialFrameOperatorCommutationRelation1}),
we have $PM_{(\nu-s)/\alpha} \, g = 0$ and therefore $PM_{s/\alpha} \, g = 0$ for $s=1,\ldots,\nu-1$.

\smallskip{}

\textbf{(i)$\boldsymbol{\Rightarrow}$(iv):}
Note that $(\gamma, \alpha \Z \times \beta \Z)$ is a frame sequence
and that $\calG = \calG(\gamma, \alpha \Z \times \beta \Z)$.
Further, Lemma~\ref{lem:StabiliserIsGroup} shows that $T_{\alpha/\nu} g \in \calG$
if and only if $\calG$ is invariant under $T_{\alpha/\nu}$,
if and only if $T_{\alpha/\nu}\gamma\in\calG$.
Let us consider the setting above with $g$ and $\gamma$ interchanged:
Define
\[
  \calF_s^*
  := \big\{
       T_{\frac{k}{\beta}} M_{\frac{\ell\nu}{\alpha}} M_{\frac{s}{\alpha}} \gamma
       \colon
       k,\ell\in\Z
     \big\},
  \qquad s=0,\ldots,\nu-1.
\]
Then, by using the implication ``(i)$\Rightarrow$(iii)'' in this setting, we get
$\calK^\ast = \calL_0^* \directSum \cdots \directSum \calL_{\nu-1}^*$,
where $\calL_s^* := \ol{\vphantom{t}\linspan}\,\calF_s^*$
and $\calK^\ast = \overline{\vphantom{t} \linspan} \, \calF^\ast$ with
$\calF^\ast := \bigl\{ T_{k/\beta} M_{\ell / \alpha} \gamma \colon k,\ell \in \Z \bigr\}$.
Note that $\calK = \calK^\ast$ by Equation~\eqref{eq:CanonicalDualOfAdjointSystem}.

We have $S_{g,\gamma} = S_{\gamma,g}^*$.
Hence, $M_{s/\alpha}P^*M_{-s/\alpha}$ is the projection onto $\calL_s^*$
with respect to the decomposition
$L^2(\R) = (\calL_0^* \directSum \cdots \directSum \calL_{\nu-1}^*)\oplus\calK^\perp$.
In particular, using the general formula $(\ker T)^\perp = \overline{\ran T^\ast}$ for a bounded
operator $T : \calH \to \calH$ (see \cite[Remarks after Theorem~II.2.19]{ConwayFA}) and
the elementary identity $(A + B)^{\perp} = A^\perp \cap B^\perp$ for subspaces $A,B \subset \calH$,
we get
\[
  \calL_0^*
  = \overline{\ran P^*}
  = (\ker P)^\perp
  = \bigl( \calL_1 \directSum \cdots \directSum \calL_{\nu-1} \bigr)^\perp \cap \calK.
  \vspace{0.15cm}
\]
For $k, \ell \in \Z$ and $s \in \{1,\ldots,\nu-1\}$ this implies
\(
  \big\<
    T_{\frac{k}{\beta}} M_{\frac{\ell\nu}{\alpha}} \gamma,
    M_{\frac{s}{\alpha}} g
  \big\>
  = 0,
\)
which is equivalent to (iv).

\smallskip{}

\textbf{(iv)$\boldsymbol{\Rightarrow}$(ii):}
For $s \in \{ 1,\ldots,\nu-1 \}$, we have
\begin{align*}
  P M_{\frac{s}{\alpha}} g
  &= \frac 1{\alpha\beta}
     \sum_{k,\ell\in\Z}
       \big\<
         M_{\frac{s}{\alpha}} g,
         T_{\frac{k}{\beta}} M_{\frac{\ell\nu}{\alpha}} \gamma
       \big\>
       T_{\frac{k}{\beta}} M_{\frac{\ell\nu}{\alpha}} g
   = 0 .
\end{align*}
Thanks to Equation~\eqref{eq:ProjectionDecomposition}, this implies $g = P_{\calK} g = Pg$.
Overall, we have thus shown $P M_{s/\alpha} \, g = \delta_{s,0} \, g$
for all $s \in \{ 0,\dots,\nu-1 \}$.
\end{proof}

Note that with $P:=(\alpha\beta)^{-1}S_{\gamma,g}$, the condition $Pf = f$ for $f\in\calL_0$
means that $(\alpha\beta)^{-1}\gamma$ is a dual window for the frame sequence
$\calF_0 = \big( g, \tfrac{1}{\beta} \Z \times \tfrac{\nu}{\alpha} \Z \big)$.
However, it is possible that $\gamma \notin \calL_0 = \overline{\vphantom{t} \linspan} \, \calF_0$.

\section{Proof of the main theorem}
\label{sec:S0}

In this section, we prove our main result, Theorem~\ref{thm:GeneralizedBLConjecture_Intro},
which we state here once more for the convenience of the reader.

\MainTheorem*

\begin{proof}
The claim is true if $\Lambda$ has rational density; see \cite[Theorem~1]{TFInvarianceAndAmalgamBL}.
Thus, assume that $\Lambda$ has irrational density $d(\Lambda) \in \R \backslash \Q$.
Write $\Lambda = A \Z^2$ with an invertible matrix $A \in \R^{2 \times 2}$.

Due to the Amalgam Balian-Low theorem \cite[Theorem~3.2]{BenedettoDifferentiationAndBLT},
it is not possible that $\calG := \calG(g,\La) = L^2(\R)$.
Hence, $\calG \neq L^2(\R)$.
Suppose towards a contradiction that $\Inv(\calG) \supsetneq \Lambda$.
According to Lemma~\ref{l:ReductionToSeparableLattice} there exist $B\in\R^{2\times 2}$
with $\det B=1$ and $\alpha,\beta > 0$ such that $\La_B = \alpha\Z\times\beta\Z$
and $T_{\frac\alpha\nu}g_B\in\calG_B$ for some $\nu\in\N_{\ge 2}$.
Set $h := g_B$ and $\calG_h := \calG_B = \calG(h,\alpha\Z\times\beta\Z)$.
Then $h\in S_0(\R)$ by \cite[Proposition~12.1.3]{GroechenigTFFoundations}
and $T_{\frac\alpha\nu}h\in\calG_h$.
Furthermore, note that with $(g,\Lambda)$, also $(h,\alpha \Z \times \beta \Z) = (g_B, \Lambda_B)$
is a Riesz sequence (cf.\ Subsection \ref{sub:SymplecticOperators}).

Let $\gamma$ be the canonical dual window for $(h, \alpha \Z \times \beta \Z)$.
By Ron-Shen duality (see \mbox{\cite[Theorem~7.4.3]{GroechenigTFFoundations}}), the adjoint system
$(h, \frac{1}{\beta} \Z \times \frac{1}{\alpha} \Z)$ is a frame for $L^2(\R)$.
Let $\gamma^{\natural}$ denote the canonical dual window
of $(h, \frac{1}{\beta} \Z \times \frac{1}{\alpha} \Z)$,
and note by Wexler-Raz orthogonality (see \cite[Theorem~7.3.1]{GroechenigTFFoundations})
that $\langle h, \gamma^{\natural} \rangle = (\alpha \beta)^{-1}$.
Next, note that \cite[Theorem~2.3]{RonShenDuality} shows
$\gamma^{\natural} = (\alpha \beta)^{-1} \gamma$ and hence
$\langle h, \gamma \rangle = \alpha \beta \cdot \langle h, \gamma^{\natural} \rangle = 1$,
which will be used below.

Since $T_{\alpha/\nu} h \in \calG_h$, Theorem~\ref{thm:AdditionalTFShiftAdjointCharacterization}
implies that $P_0 := (\alpha \beta)^{-1} S_{\frac{1}{\beta}\Z \times \frac{\nu}{\alpha} \Z, \gamma, h}$
is an idempotent (i.e., $P_0^2 = P_0$).
We now wish to apply Theorem~\ref{thm:TraceOfProjectionsVerySpecial} to derive a contradiction.
To this end, first note that
\(
  \alpha \beta
  = \big( d(\Lambda_B) \big)^{-1}
  = |\det B A|
  = |\det A|
  = \big( d(\Lambda) \big)^{-1}
  \in \R \backslash \Q.
\)
Next, set $U := M_{\beta}$ and $V := T_{\frac{\alpha}{\nu}}$.
A direct calculation shows that
\[
  UV = e^{2\pi i\theta} V U,
  \qquad \text{where} \qquad
  \theta := \tfrac{\alpha \beta}{\nu} \in \R \backslash \Q .
\]
Note that also $\gamma \in S_0(\R)$; see \cite[Theorem~7]{BalanDensityOvercompletenessLocalization}.
Hence, we may use Equation~\eqref{e:S0_CFO} and obtain
\begin{equation}\label{e:projection}
  P_0 = \frac{1}{\nu}
    \sum_{m,n\in \Z}
        \big\<
          h,
          T_{\frac{m\alpha}{\nu}}M_{n\beta}\gamma
        \big\>
        T_{\frac{m\alpha}{\nu}}
        M_{n\beta}
   = \sum_{m,n \in \Z}
         \frac 1\nu
         \big\<h, V^mU^n\gamma\big\>
         \,V^m U^n
       ,
\end{equation}
with coefficient sequence
\(
  a
  =  (a_{m,n})_{m,n\in \Z}
  := \big(
       \tfrac 1\nu \< h, V^m U^n\gamma\>
     \big)_{m,n\in \Z}\in \ell^1(\Z^2).
\)
Therefore, Theorem~\ref{thm:TraceOfProjectionsVerySpecial} shows that
$\tfrac{1}{\nu} = a_{0,0} \in \Z + \theta \Z$, say $\tfrac 1\nu = m + n \theta$ for some $m,n\in\Z$.
We must have $n\neq 0$, since otherwise $\tfrac 1\nu = m \in \Z$,
in contradiction to $\nu\geq 2$. Thus, $\theta = \frac{1}{n \nu} - \frac{m}{n} \in \Q$,
which is the desired contradiction, since $\theta = \frac{\alpha \beta}{\nu}$ is irrational.
\end{proof}

\begin{rem}
  On a first look, it might appear as if the proof of Theorem~\ref{thm:GeneralizedBLConjecture_Intro}
  would also apply in case of $g \in \HH^1$:
  First, the classical Balian-Low theorem implies that $\calG :=\calG(g,\Lambda) \subsetneq L^2(\R)$,
  so that Lemma \ref{l:ReductionToSeparableLattice} allows the reduction
  to a Gabor Riesz sequence $(h,\La)$ with $h\in\HH^1$,
  a separable lattice $\La = \alpha\Z\times\beta\Z$, and an additional time-frequency shift
  of the form ${T_{\alpha/\nu} \, h \in \calG}$.
  One can then apply Theorem~\ref{thm:AdditionalTFShiftAdjointCharacterization} to see that
  that ${L^2(\R) = \calL_0 \directSum \cdots \directSum \calL_{\nu-1}}$.
  In the $S_0$-case, we then employed Janssen's representation \eqref{e:projection} for the projection
  ${P_0 = (\alpha\beta)^{-1}S_{\frac{1}{\beta} \Z \times \frac{\nu}{\alpha} \Z, \gamma,h}}$,
  which then led to success in the proof of Theorem~\ref{thm:GeneralizedBLConjecture_Intro},
  thanks to existing results concerning the structure of the irrational rotation algebra.
  However, in the case $h \in \HH^1$ the series in \eqref{e:projection} might not converge
  \emph{in operator norm}, so that one does not know whether $P_0$ belongs to the
  irrational rotation algebra.
  Thus, the proof breaks down at this point.
\end{rem}

\section*{Acknowledgments}
A. Caragea acknowledges support by the DFG Grant PF 450/11-1.
D.G.\ Lee acknowledges support by the DFG Grants PF 450/6-1 and PF 450/9-1.
F.\ Voigtlaender acknowledges support by the DFG in the context of the Emmy Noether junior research group VO 2594/1-1.

The authors take pleasure in thanking Karlheinz Gröchenig for suggesting the idea of considering
the trace on the irrational rotation algebra.
In fact, we had established the characterization in
Theorem~\ref{thm:AdditionalTFShiftAdjointCharacterization}, but were unable to
prove that a projection as in \eqref{e:projection} cannot exist,
until Karlheinz Gröchenig suggested to us at SampTA~2019 that we should
try to use the trace on the irrational rotation algebra.
Without his hint, we would probably not have managed to prove our main result,
Theorem~\ref{thm:GeneralizedBLConjecture_Intro}.

In addition, the authors would like to thank Radu Balan, Ilya Krishtal, Götz E.~Pfander,
and Jordy van Velthoven for fruitful discussions and hints.

\appendix

\section{Appendix}\label{sec:Appendix}

In this section, we make use of a deep result concerning the structure of the
\emph{irrational rotation algebra} $\calA_\theta$ (see
\cite{DavidsonCStarAlgebrasByExample,PimsnerVoiculescuIrrationalRotationAlgebra,
RieffelIrrationalRotationAlgebra})
to prove the following auxiliary statement, which is a crucial ingredient for the
proof of our main result, Theorem~\ref{thm:GeneralizedBLConjecture_Intro}.
As usual, we denote the set of all bounded linear operators from a Hilbert space $\calH$ into
itself by $\calB(\calH)$.

\begin{thm}\label{thm:TraceOfProjectionsVerySpecial}
  Let $\calH \neq \{ 0 \}$ be a Hilbert space and let ${U,V\in\calB(\calH)}$ be unitary
  and such that $U V = e^{2 \pi i \theta} V U$, for some $\theta \in \R \backslash \Q$.
  If $a = (a_{k,\ell})_{k,\ell \in \Z} \in \ell^1(\Z^2)$ is such that the operator
  $P_a := \sum_{k,\ell \in \Z} a_{k,\ell} V^k U^\ell$ satisfies
  $P_a^2 = P_a$, then $a_{0,0} \in \Z + \theta \Z$.
\end{thm}

The proof will make use of some parts of the theory of $C^\ast$-algebras,
which we recall here for the convenience of the reader, based on \cite{MurphyCStarAlgebras}.
Readers familiar with $C^\ast$-algebras will probably want to skip this part%
---except possibly Lemma~\ref{l:similarity}.

A $C^\ast$-algebra is a (complex) Banach algebra $(A, \| \cdot \|)$,
additionally equipped with a map ${A \to A, x \mapsto x^\ast}$ (called the \emph{involution} on $A$),
satisfying the following properties:
\begin{itemize}
  \item $(x + y)^\ast = x^\ast + y^\ast$,
        $(\lambda \, x)^\ast = \overline{\lambda} \, x^\ast$,
        and $(x \, y)^\ast = y^\ast \, x^\ast$ and $(x^\ast)^\ast = x$
        for all $x,y \in A$ and $\lambda \in \CC$;

  \item $\| x^\ast \| = \| x \|$ and $\| x^\ast x \| = \| x \|^2$ for all $x \in A$.
\end{itemize}
An element $p\in A$ is called an {\em idempotent} if $p^2=p$.
An idempotent $p$ is called a \emph{projection} if additionally $p = p^\ast$ holds.
A $C^\ast$-algebra $A$ is called \emph{unital} if it contains a (necessarily unique)
element $1 \in A$ satisfying $1 \neq 0$ and $x \, 1 = 1 \, x = x$ for all $x \in A$.
In a unital $C^\ast$-algebra $A$, an element $x \in A$ is called
\emph{unitary} if $x^\ast x = 1 = x x^\ast$.
If $A$ is a unital $C^\ast$-algebra and $a\in A$, then $\sigma(a^*a)\subset [0,\infty)$;
see \cite[Theorem~2.2.4]{MurphyCStarAlgebras}.
Here, $\sigma(b) = \{ \lambda \in \CC \colon b - \lambda 1 \text{ not invertible in } A \}$.

\begin{lem}\label{l:similarity}
  Any idempotent $e$ in a unital $C^*$-algebra $A$ is similar to a projection $p\in A$.
  That is, there exist a projection $p \in A$ and an invertible element $a \in A$
  such that $e = a^{-1} p a$.
\end{lem}

\begin{proof}
We set $b := e^* - e$ and $z := 1 + b^* b$.
Note that $z$ is invertible since $\sigma (b^\ast b) \subset [0,\infty)$.
We have
\[
  ez
  = e + (e - e e^*)(e^* - e)
  = e e^* e
  = e + (e - e^*)(e^* e - e)
  = z e.
\]
Consequently, $e z^{-1} = z^{-1} e$ and, as $z = z^*$, also $e^* z^{-1} = z^{-1} e^*$.
Now, define the element $p := e z^{-1} e^*$.
We have $p^* = p$.
Furthermore, since we just saw that $z^{-1}$ commutes with $e$ and $e^\ast$
and that $e e^\ast e = z e$, we also see that ${p^2 = z^{-2} (e e^* e) e^* = z^{-1} e e^* = p}$.
Hence, $p$ is a projection.
We further observe that $ep = p$
and ${p e = e z^{-1} e^\ast e = z^{-1} e e^\ast e = z^{-1} z e = e}$.
Set $a := 1 - p + e$.
Then we see because of
\[
  (1 \mp p \pm e)(1 \pm p \mp e)
  = 1 \pm p \mp e \mp p - p + e \pm e + p - e
  = 1
\]
that $a$ is invertible with $a^{-1} = 1+p-e$.
Hence, from $ae = e-pe+e = e$ we obtain
\[
  a e a^{-1}
  = e(1 + p - e)
  = e + e p - e
  = e p
  = p,
\]
which proves the lemma.
\end{proof}

A closed subspace $B$ of a $C^*$-algebra $A$ is called a {\em $C^*$-subalgebra} of $A$
if it is closed under both multiplication and involution.
It is clear that $B$ is then itself a $C^*$-algebra.
As usual, given a subset $S \subset A$, there is a smallest (with respect to inclusion)
$C^*$-subalgebra of $A$ containing $S$.
We call it the \emph{$C^\ast$-algebra generated by $S$}, and denote it by $C^\ast(S)$.

A map $\varphi : A \to B$ between two $C^\ast$-algebras $A$ and $B$ is called a $\ast$-homomorphism
if it is linear and satisfies $\varphi(x \, y) = \varphi(x) \, \varphi(y)$ as well as
$\varphi(x^\ast) = [\varphi(x)]^\ast$ for all $x, y \in A$.
A bijective $\ast$-homomorphism is called a $\ast$-isomorphism.
Any $\ast$-homomorphism ${\varphi : A \to B}$ necessarily
satisfies $\| \varphi(x) \|_B \leq \| x \|_A$ for all $x \in A$, and is hence continuous;
see \mbox{\cite[Theorem~2.1.7]{MurphyCStarAlgebras}}.

\begin{proof}[Proof of Theorem \rmref{thm:TraceOfProjectionsVerySpecial}]
  We will make use of the so-called \emph{irrational rotation algebra} $\calA_\theta$,
  as introduced for instance in \mbox{\cite[Chapter~VI]{DavidsonCStarAlgebrasByExample}}.
  The actual definition of this algebra is not relevant for us;
  we will only need to know that it satisfies the following properties:
  \begin{itemize}
    \item $\calA_\theta$ is a unital $C^\ast$-algebra;
          \vspace{0.1cm}

    \item The algebra $\calA_\theta$ is \emph{universal} among all unital $C^\ast$-algebras
          generated by unitary elements $U,V$ satisfying $U V = e^{2 \pi i \theta} V U$.
          Thus, defining $\calA := C^\ast(U,V)$ as a $C^*$-subalgebra of $\calB(\calH)$ with $U,V$
          as in the statement of Theorem~\ref{thm:TraceOfProjectionsVerySpecial},
          there is a $\ast$-isomorphism $\varphi : \calA \to \calA_\theta$;
          this follows from \cite[Theorem~VI.1.4]{DavidsonCStarAlgebrasByExample}.
          \vspace{0.1cm}

    \item As shown in \cite[Corollary~VI.1.2 and Proposition~VI.1.3]{DavidsonCStarAlgebrasByExample},
          there is a unique {\em \braces{unital\,} trace} $\tau : \calA_\theta \to \CC$.
          By definition of a trace, this means in particular that $\tau$
          is linear and continuous, satisfying $\tau(1) = 1$
          and $\tau(x y) = \tau (y x)$ for all $x,y \in \calA_\theta$.%
          \vspace{0.1cm}

    \item For any projection $p \in \calA_\theta$, we have $\tau(p) \in \Z + \theta \Z$;
          see \cite[Theorem~1.2]{RieffelIrrationalRotationAlgebra}.
          We remark that this result was originally proven
          in \cite{PimsnerVoiculescuIrrationalRotationAlgebra}.
  \end{itemize}

  Let us define $\tau^{\natural} := \tau \circ \varphi$, and note that
  $\tau^{\natural} : \calA \to \CC$ is continuous.
  It is easy to see that $\tau^{\natural}$ is linear with
  $\tau^{\natural}(\mathrm{id}_{\calH}) = 1$ and $\tau^\natural (A B) = \tau^{\natural} (B A)$
  for all $A,B \in \calA$; this is called the \emph{cyclicity} of the trace.
  Next, from the relation $U V = e^{2 \pi i \theta} V U$,
  we immediately get for $k,\ell \in \Z$ that
  \[
    V^k U^\ell
    = e^{-2\pi i\ell\theta} V^{k-1} U^\ell V
    = e^{-2\pi ik\theta} UV^k U^{\ell-1}.
  \]
  Thus, noting that $V^kU^\ell\in\calA$, we obtain
  \(
    \tau^{\natural}(V^k U^\ell)
    = e^{-2\pi i\ell\theta}\tau^{\natural}(V^k U^\ell)
    = e^{-2\pi ik\theta}\tau^{\natural}(V^k U^\ell)
  \)
  by cyclicity.
  As $\theta$ is irrational, this implies $\tau^{\natural}(V^k U^\ell) =\delta_{\ell,0}\delta_{k,0}$.
  Next, since we have $\| V^k U^\ell \| = 1$ for all $k, \ell \in \Z$ and since
  $a \in \ell^1(\Z^2)$, we see that $P_a = \sum_{k,\ell \in \Z} a_{k,\ell} V^k U^\ell \in \calA$,
  with unconditional convergence of the defining series.
  Hence,
  \[
    \tau^{\natural}(P_a)
    = \sum_{k,\ell \in \Z}
        a_{k,\ell} \cdot \tau^{\natural} (V^k U^\ell)
    = a_{0,0}.
  \]
  Since $P_a^2 = P_a$ and since $\varphi : \calA \to \calA_\theta$
  is a $\ast$-homomorphism, we see that ${e := \varphi(P_a) \in \calA_\theta}$ is an idempotent.
  By Lemma \ref{l:similarity} there exist $b,p\in\calA_\theta$ such that $b$ is invertible,
  $p$ is a projection, and $e = b^{-1} p b$.
  Thanks to the cyclicity of the trace, we thus see that
  ${a_{0,0} = \tau^{\natural}(P_a) = \tau(e) = \tau(b^{-1} p b) = \tau(p)\in \Z + \theta \Z}$,
  as claimed.
\end{proof}

\vspace*{.3cm}
\section*{Author Affiliations}

\begin{thebibliography}{1}

\bibitem{BalanDensityOvercompletenessLocalization}
R.~Balan, P.G.~Casazza, C.~Heil, and Z.~Landau,
\emph{Density, overcompleteness, and localization of frames. II. Gabor systems},
J. Fourier Anal. Appl. 12(3) (2006), 309--344.

\bibitem{BenedettoDifferentiationAndBLT}
J.J.~Benedetto, C.~Heil, and D.F.~Walnut,
\emph{Differentiation and the Balian-Low theorem},
J. Fourier Anal. Appl. 1 (1995), 355--402.

\bibitem{TFInvarianceIntegerLattices}
C.~Cabrelli, D.G.~Lee, U.~Molter, and G.E.~Pfander,
\emph{Time-frequency shift invariance of Gabor spaces generated by integer lattices},
J. Math. Anal. Appl. 474 (2019), 1289--1305.

\bibitem{TFInvarianceAndAmalgamBL}
C.~Cabrelli, U.~Molter, and G.E.~Pfander,
\emph{Time-frequency shift invariance and the Amalgam {B}alian-{L}ow theorem},
Appl. Comput. Harmon. Anal. 41 (2016), 677--691.

\bibitem{AmalgamBLForSymplecticLatticesRationalDensity}
C.~Cabrelli, U.~Molter, and G.E.~Pfander,
\emph{An Amalgam Balian-Low Theorem for symplectic lattices of rational density},
Proceedings International Conference on Sampling Theory and Applications,
Washington DC, 2015.


\bibitem{BLForSubspaces}
A.~Caragea, D.G.~Lee, G.E.~Pfander, and F.~Philipp,
\emph{A Balian-Low theorem for subspaces},
J. Fourier Anal. Appl. 25 (2019), 1673--1694.


\bibitem{QuantitativeSubspaceBL}
A.~Caragea, D.G.~Lee, F.~Philipp, and F.~Voigtlaender,
\emph{A quantitative subspace Balian-Low theorem},
Appl. Comput. Harmon. Anal. 55 (2021) 368--404.

\bibitem{ChristensenBook}
O.~Christensen,
\emph{An introduction to frames and Riesz bases},
Birkh\"{a}user/Springer, [Cham], 2016.


\bibitem{ConwayFA}
J.~B.~Conway,
\emph{A course in functional analysis},
second edition,
Springer-Verlag, New York, 1990.


\bibitem{DaubechiesBalianLow}
I.~Daubechies,
\emph{The wavelet transform, time-frequency localization and signal analysis},
IEEE Trans.~Inform.~Theory 36(5) (1990), 961--1005.


\bibitem{DavidsonCStarAlgebrasByExample}
K.~R.~Davidson,
\emph{$C^\ast$-algebras by example},
American Mathematical Society, Providence, RI, 1996.


\bibitem{FeichtingerNewSegalAlgebra}
H.G.~Feichtinger,
\emph{On a new Segal algebra},
Monatsh.~Math.~92 (1981), 269--289.

\bibitem{FeichtingerGroechenigGaborFramesAndTFAnalysisDistributions}
H.G.~Feichtinger and K.~Gr\"ochenig,
\emph{Gabor frames and time-frequency analysis of distributions},
J. Funct. Anal. 146 (1997), 464--495.





\bibitem{GroechenigTFFoundations}
K.~Gr\"ochenig,
\emph{Foundations of time-frequency analysis},
Birkh\"auser, Boston, Basel, Berlin, 2001.

\bibitem{GroechenigLeinert}
K.~Gr\"ochenig and M.~Leinert,
\emph{Wiener's lemma for twisted convolution and Gabor frames},
J. Am. Math. Soc. 17 (2003), 1--18.




\bibitem{HeilTinaztepe}
C. Heil and R. Tinaztepe,
\emph{Modulation spaces, BMO, and the Balian-Low theorem},
Sampl. Theory Signal Image Process. 11 (2012), 25--41.

\bibitem{HewittRoss}
E. Hewitt and K.A. Ross,
{\em Abstract harmonic analysis,
Volume I: Structure of topological groups, Integration Theory, Group Representations},
second edition,
Springer-Verlag, 1963.

\bibitem{jacobson}
N. Jacobson,
{\em Basic Algebra I},
second edition,
W.H. Freeman and Company, New York, 1985.

\bibitem{JakobsenNoLongerNewSegalAlgebra}
M.S.~Jakobsen,
\emph{On a \braces{no longer} new Segal algebra: A review of the Feichtinger algebra},
J.~Fourier Anal. Appl.~24 (2018), 1579--1660.

\bibitem{j}
A.J.E.M. Janssen,
\emph{Duality and biorthogonality for Weyl-Heisenberg frames},
J. Fourier Anal. Appl. 1 (1995), 403--436.







\bibitem{LeoniSobolevSpaces}
G.~Leoni,
\emph{A first course in Sobolev spaces},
second edition,
American Mathematical Society, 2017.


\bibitem{macduffee}
C.C. MacDuffee,
\emph{The theory of matrices},
Chelsea Publishing Company, New York, 1946.

\bibitem{MurphyCStarAlgebras}
G.~J.~Murphy,
\emph{$C^\ast$-algebras and operator theory},
Academic Press, Inc., Boston, MA, 1990.



\bibitem{PimsnerVoiculescuIrrationalRotationAlgebra}
M.~Pimsner and D.~Voiculescu,
\emph{Imbedding the irrational rotation $C^\ast$-algebra into an AF-algebra},
J.~Operator Theory 4(2) (1980), 201--210.


\bibitem{RieffelIrrationalRotationAlgebra}
M.~A.~Rieffel,
\emph{$C^\ast$-algebras associated with irrational rotations},
Pacific J.~Math. 93(2) (1981), 415--429.

\bibitem{RonShenDuality}
A.~Ron and Z.~Shen,
\emph{Weyl-Heisenberg frames and Riesz bases in $L_2(\R^d)$},
Duke Math. J. 89 (1997), 237--282.



\bibitem{RudinRealAndComplexAnalysis}
W.~Rudin,
\emph{Real and complex analysis},
3rd ed., McGraw-Hill, Inc., 1987.



\end{thebibliography}
\end{document}